\documentclass[11pt,reqno]{amsart}
\usepackage[margin=1in]{geometry}
\usepackage{hyperref, amssymb, amsfonts, amsthm, amsrefs, mathtools, enumerate, verbatim, color, makecell}

\newtheorem{Theorem}{Theorem}[section]
\newtheorem{Proposition}[Theorem]{Proposition}
\newtheorem{Lemma}[Theorem]{Lemma}
\newtheorem{Corollary}[Theorem]{Corollary}
\newtheorem{Claim}[Theorem]{Claim}

\theoremstyle{definition}
\newtheorem{Definition}[Theorem]{Definition}

\newcommand{\rca}{\mathsf{RCA}_0}
\newcommand{\wkl}{\mathsf{WKL}_0}
\newcommand{\rrt}{\mathsf{RRT}}

\newcommand{\iso}{\mathsf{I}\Sigma^0_1}
\newcommand{\ipo}{\mathsf{I}\Pi^0_1}
\newcommand{\bso}{\mathsf{B}\Sigma^0_1}
\newcommand{\bpo}{\mathsf{B}\Pi^0_1}
\newcommand{\bst}{\mathsf{B}\Sigma^0_2}

\newcommand{\mlr}{\mathsf{MLR}}
\newcommand{\wtr}{\mathsf{W2R}}
\newcommand{\dnr}{\mathsf{DNR}}
\newcommand{\dnrr}[1]{#1\mbox{-}\mathsf{DNR}}
\newcommand{\mlrr}[1]{#1\mbox{-}\mathsf{MLR}}

\newcommand{\cinc}{C\mbox{-}\mathsf{INC}}
\newcommand{\cran}{\mathsf{CR}}

\newcommand{\sran}{\mathsf{SR}}

\newcommand{\bran}{\mathsf{BR}}
\newcommand{\wdrr}[1]{#1\mbox{-}\mathsf{WDR}}

\newcommand{\pam}{\mathsf{PA^{-}}}

\DeclareMathOperator{\dom}{\mathrm{dom}}

\newcommand{\andd}{\wedge}

\newcommand{\la}{\langle}
\newcommand{\ra}{\rangle}
\newcommand{\da}{{\downarrow}}
\newcommand{\ua}{{\uparrow}}
\newcommand{\imp}{\rightarrow}
\newcommand{\Imp}{\Rightarrow}
\newcommand{\biimp}{\leftrightarrow}
\newcommand{\Biimp}{\Leftrightarrow}
\newcommand{\Nb}{\mathbb{N}}
\newcommand{\Qb}{\mathbb{Q}}
\newcommand{\Rb}{\mathbb{R}}

\newcommand{\Vb}{\mathbb{V}}
\newcommand{\rst}{{\restriction}}
\newcommand{\smf}{\smallfrown}

\newcommand{\leqT}{\leq_\mathrm{T}}
\newcommand{\leT}{<_\mathrm{T}}
\newcommand{\geqT}{\geq_\mathrm{T}}
\newcommand{\geT}{>_\mathrm{T}}

\newcommand{\leqm}{\leq_\mathrm{m}}

\newcommand{\leqtt}{\leq_\mathrm{tt}}

\newcommand{\leqwtt}{\leq_\mathrm{wtt}}

\newcommand{\leL}{<_\mathrm{L}}

\newcommand{\MP}[1]{\mathcal{#1}}
\newcommand{\mbf}[1]{\mathbf{#1}}
\newcommand{\mf}[1]{\mathfrak{#1}}

\usepackage{xcolor}	
\usepackage{soul}

\AtBeginDocument{%
   \def\MR#1{}
}

\title{Randomness notions and reverse mathematics}

\author{Andr\'{e} Nies}
\address{
Department of Computer Science\\
University of Auckland\\
Private Bag 92019\\
Auckland\\
New Zealand\\
}
\email{andre@cs.auckland.ac.nz}
\urladdr{https://www.cs.auckland.ac.nz/~nies/}

\author{Paul Shafer}
\address{
School of Mathematics\\
University of Leeds\\
Leeds\\
LS2 9JT\\
UK}
\email{p.e.shafer@leeds.ac.uk}
\urladdr{http://www1.maths.leeds.ac.uk/~matpsh/}

\date{\today}

\begin{document}

\begin{abstract}
We investigate the strength of a randomness notion $\mathcal  R$ as  a set-existence principle in second-order arithmetic: for each $Z$ there is an $X$ that is $\mathcal R$-random relative to $Z$.  We show that the equivalence between $2$-randomness and being infinitely often $C$-incompressible is provable in $\rca$.  We verify that $\rca$ proves the basic implications  among randomness notions: $2$-random $\Rightarrow$ weakly  $2$-random $\Rightarrow$  Martin-L\"of random $\Rightarrow$  computably random $\Rightarrow$   Schnorr random.  Also, over $\rca$ the existence of computable randoms is equivalent to the existence of Schnorr randoms.  We show that the existence of balanced randoms is equivalent to the existence of Martin-L\"of randoms, and we describe a sense in which this result is nearly optimal.  
\end{abstract}

\maketitle

\section{Introduction}
\subsection*{Randomness} The theory of randomness via algorithmic tests has its beginnings in Martin-L\"of's paper  \cite{MartinLof:68}, in the work of Schnorr~\cites{Schnorr:71,Schnorr:75}, as well as in the work of Demuth such as~\cite{Demuth:75a}. Each of these authors employed  algorithmic tools to  introduce  tests of whether  an infinite bit sequence is random. Rather than an absolute notion of algorithmic randomness, a hierarchy of randomness notions emerged based on the strength of  the algorithmic tools that were allowed.  Martin-L\"of introduced the randomness notion now named after him, which was based on uniformly computably enumerable sequences of open sets in Cantor space. Schnorr considered more restricted tests based on computable betting strategies, which led to the weaker notion now called computable randomness and the even weaker notion now called Schnorr randomness.  Notions of randomness stronger than Martin-L\"of's but still arithmetical were introduced somewhat later by Kurtz \cite{Kurtz:81}.  Of importance for us will be   2-randomness (namely, ML-randomness relative to the halting problem), and the  notion of weak 2-randomness intermediate between 2-randomness and ML-randomness.   See Sections~\ref{sec-RandFormal}  and~\ref{sec-Implications} for the formal definitions.  

 The field of algorithmic randomness entered a   period of intense activity from   the late 1990s, with a flurry of research papers leading to the publication of two textbooks~\cites{NiesBook, DHBook}.  One reason for this was the realization, going back  to Ku\v{c}era~\cites{Kucera:85,Kucera:86},  that sets satisfying  randomness notions interact in a meaningful way with the computational complexity of Turing  oracles (the latter is a prime topic in computability theory).   One can discern two main directions  in  the study of randomness notions: 
 
 \medskip 
 \noindent \emph{(A) Characterizing  theorems.} Give conditions on bit sequences that are equivalent to being random in a particular sense, and thereby reveal more about the randomness notions. The   Levin-Schnorr theorem  characterizes ML-randomness of $Z$  by the incompressibility of $Z$'s initial segments in the sense of the prefix-free descriptive string complexity $K$:  $$Z \text{ is ML-random } \Leftrightarrow \,  \exists b \forall n \, K(Z \rst n) \ge n-b.$$  2-randomness is equivalent to being infinitely often incompressible in the sense of the plain descriptive string complexity $C$ 
  \cites{NiesStephanTerwijn,MillerKRand} (see also~\cite[Theorem~3.6.10]{NiesBook}):  $$Z \text{ is 2-random }   \Leftrightarrow \,  \exists b \exists^\infty  n \, C(Z \rst n) \ge n-b.$$ There are also  examples of characterizations not relying on the descriptive complexity of  initial segments. For instance, a bit sequence $Z$ is  2-random iff $Z$  is ML-random and the halting probability $\Omega$ is ML-random relative to it; $Z$ is weakly 2-random iff $Z$ is ML-random and bounds no incomputable set that is below the halting problem.

  \medskip

 \noindent \emph{(B) Separating theorems.} Given randomness notions that appear to be close to each other, one wants to find a bit sequence that is random in the weaker sense but not in the stronger sense. For instance, Schnorr provided a sequence that is Schnorr random but not computably random.   For a more recent example, Day and Miller \cite{Day.Miller:15} separated notions only slightly stronger than ML-randomness.  They  provided a sequence that is difference random but not density random. Difference randomness, introduced via so-called difference tests,  is equivalent to being  ML-random and Turing incomplete~\cite{Franklin.Ng:10}. Density randomness, by definition, is the combination of   ML-randomness and satisfying the conclusion of the Lebesgue density theorem for effectively closed sets.   
 
 Some separations of   randomness notions  are open problems. For instance, it is unknown whether  Oberwolfach randomness is stronger than density randomness~\cite[Section 6]{Miyabe.Nies.Zhang:16} and of course whether ML-randomness is stronger than Kolmogorov-Loveland randomness \cites{Ambos.Kucera:00,Miller.Nies:06}. 

Some motivation for obtaining separations of    notions that appear to be close was provided by the above mentioned  interaction of randomness with computability and, in particular, with  lowness  properties of oracles. The Turing incomplete ML-random set obtained in  the Day/Miller result  is Turing  above all $K$-trivial sets  because it is not  density random~\cite{Bienvenu.Day.ea:14}.  Whether such a set exists had been open for eight years \cite{Miller.Nies:06}.

\subsection*{The viewpoint of reverse mathematics} Our purpose is to study randomness notions from the viewpoint of  {reverse mathematics}. This program  in the foundations of mathematics,  introduced by H.\ Friedman~\cite{Friedman}, attempts to  classify the axiomatic strength of mathematical theorems.  The typical goal in reverse mathematics is to determine which axioms are necessary and sufficient to prove a given mathematical statement.  In order to do this, one fixes a base axiom system over which the reasoning is done. Then one asks which  stronger axioms must be added to this base system in order to prove a given statement.

Algorithmic randomness  plays an important role in reverse mathematics.  Axioms asserting that random sets exist are interesting because  typically they are weak compared to traditional comprehension schemes; in particular, the randomness notions that we consider produce axioms that are weaker than arithmetical comprehension. They still have important mathematical consequences, particularly concerning measure theory.    Given a   randomness notion~$\mathcal R$, we consider the statement ``for every set $Z$, there is a set $X$ that is $\mathcal R$-random relative to $Z$.'' Informally, we refer to this statement as the ``existence of $\mathcal R$-random sets.''  

Quite a bit is known in the case that  $\mathcal R$ is Martin-L\"of randomness. The existence of Martin-L\"of random sets is equivalent to \emph{weak weak K\"onig's lemma}, which states that every binary-branching tree of positive measure has an infinite path.  This equivalence is obtained by formalizing a classic result of Ku\v{c}era (see e.g.~\cite[Proposition~3.2.24]{NiesBook}).  
  Via the equivalence, the existence of Martin-L\"of random sets is also equivalent to the statement ``every Borel measure on a compact complete separable metric space is countably additive''~\cite{YuSimpson}, as well as  to the monotone convergence theorem for Borel measures on compact metric spaces~\cite{YuLebesgue} (see also~\cite[Section~X.1]{SimpsonSOSOA}).
  Recently~\cite{NiesTriplettYokoyama}, equivalences between the existence of Martin-L\"of random sets and well-known theorems from   analysis have been found:   ``every continuous function of bounded variation is differentiable at some point'' and   ``every continuous function of bounded variation is differentiable almost everywhere.''

In this paper we mainly consider the reverse mathematics of randomness notions other than Martin-L\"of's.  The two directions outlined above   lead to  two   types  of questions.

\medskip 
\begin{itemize} \item[(A)] Examine   whether characterizing  theorems   can be proved  over  a weak axiomatic system such as $\rca$.  

\item[(B)] For   randomness notions that appear close to each other yet  can be separated,  see whether nonetheless the corresponding existence principles are equivalent over $\rca$.  (If so, this would gives a precise  meaning to the intuition that the notions are close.)
\end{itemize}

\medskip

\subsection*{Results}  Our first result follows direction (A):  we investigate the above-mentioned fact that a set is $2$-random if and only if it is infinitely often $C$-incompressible.   Formalizing randomness notions relative to the halting problem is delicate in weak axiomatic systems because of subtleties involving the induction axioms.  For example, the existence of $2$-random sets does not imply $\Sigma^0_2$-bounding (equivalently, $\Delta^0_2$-induction)~\cite{SlamanFrag}, but weak weak K\"onig's lemma for $\Delta^0_2$ trees (i.e., \emph{$2$-weak weak K\"onig's lemma}) does imply $\Sigma^0_2$-bounding~\cite{AvigadDeanRute}.  Therefore the equivalence between the existence of $2$-random sets and $2$-weak weak K\"onig's lemma requires $\Sigma^0_2$-bounding.  It is  then  natural to ask how much induction is required to prove theorems about $2$-random sets. We show that the equivalence between $2$-randomness and infinitely often $C$-incompressibility can be proved without appealing to $\Sigma^0_2$-bounding.  Formalizing infinitely often $C$-incompressibility in weak systems is straightforward, whereas formalizing $2$-randomness in terms of tests is not, so we hope that the formalized equivalence between the two notions will prove useful in future applications of algorithmic randomness in reverse mathematics, in addition to being technically interesting.
  
Towards direction (B), as a motivating example  consider    balanced randomness,  introduced in   \cite[Section 7]{FigueiraHirschfeldtMillerNgNies} (see  Definition~\ref{df:balanced} below),  which was   the   first notion   slightly  stronger than ML-randomness considered (Oberwolfach, density, and difference randomness, discussed above,  are  even closer to ML-randomness).  
The existence of Martin-L\"of random sets is equivalent to the existence of balanced random sets (Theorem~\ref{thm-MLRimpBran} below). Always relative to some oracle,  if  a balanced random set exists, then that set is Martin-L\"of random; conversely, if a Martin-L\"of random set exists, then at least one of its ``halves'' (i.e., either the bits in the even positions or the bits in the odd positions)  is balanced random, so a balanced random set exists.  

  We show in Theorem~\ref{thm-WKLnothWDR}  that the preceding equivalence is nearly optimal, in the sense that if $h \colon \Nb \imp \Nb$ is any function that eventually dominates every function of the form $n \mapsto k^n$, then the existence of $h$-weakly Demuth random sets is strictly stronger than the existence of Martin-L\"of random sets.

Still following (B), we show that 
 the existence of Schnorr random sets is equivalent to the existence of computably random sets (Theorem~\ref{thm-SRimpCR}).
In all cases, we actually prove that the equivalence holds for the same oracle. 

 In the alternative context of the    Muchnik and Medvedev degrees (see~\cites{HinmanSurvey, SimpsonSurvey, SimpsonMassProblemsRandomness} for background), related work has recently been done by Miyabe~\cite{Miyabe}. He views randomness notions     as mass problems (so  there is no relativization). Miyabe shows that computable randomness and Schnorr randomness are Muchnik equivalent but not Medvedev equivalent, and he gives a similar result for difference randomness versus ML-randomness.  Yet another alternative context for (B) is given by the Weihrauch degrees. Randomness notions are now viewed as multivalued functions mapping an oracle $X$ to the sets random in $X$.  See~\cites{Brattka2015LasVC, brattka2016algebraic, Brattka2017UniformComputability}.  In the Weihrauch degrees, ML-randomness is strictly weaker than weak weak K\"onig's lemma.  Brattka and Pauly~\cite[Proposition~6.6]{brattka2016algebraic} exactly characterizes ML-randomness in terms of weak weak K\"onig's lemma and a weak choice principle.

This paper is organized as follows.  In Section~\ref{sec-RM}, we recall the basic axiom system $\rca$ and establish notational conventions.  In Sections~\ref{sec-RandFormal} and \ref{sec-Implications},  we explain how the randomness notions we discussed above  can be  formalized in second-order arithmetic.    In Section~\ref{sec-Inc}, we show that the equivalence between $2$-randomness and infinitely often $C$-incompressibility can be proved in $\rca$.    In Sections~\ref{sec-Implications} and~\ref{s:wD}, we study implications and equivalences among randomness notions as set-existence principles that can be proved in $\rca$.  In Section~\ref{sec-NonImp}, we exhibit non-implications over $\rca$ among certain randomness notions, recursion-theoretic principles, and combinatorial principles. 

\section{Preliminaries}\label{sec-RM}

\subsection*{Basic axioms}
  We provide  a short  introduction to the typical base system of reverse mathematics $\rca$ that suits our purposes here. We refer  the reader to    Simpson~\cite{SimpsonSOSOA} for further details.
The setting of $\rca$ is second-order arithmetic. Its axioms consist of:
\begin{itemize}
\item  
 The basic axioms  of Peano arithmetic (denoted $\pam$) expressing that the natural numbers form a discretely-ordered commutative semi-ring with $1$;

\medskip

\item the \emph{$\Sigma^0_1$ induction scheme} ($\iso$, for short), which consists of the universal
closures of all formulas of the form
\begin{align*}
\bigl(\varphi(0) \andd \forall n(\varphi(n) \imp \varphi(n+1))\bigr) \imp \forall n \varphi(n),\label{fmla-IND}\tag{$\star$}
\end{align*}
where $\varphi$ is $\Sigma^0_1$;

\medskip

\item the \emph{$\Delta^0_1$ comprehension scheme}, which consists of the universal closures of all formulas of the form
\begin{align*}
\forall n (\varphi(n) \biimp \psi(n)) \imp \exists X \forall n(n \in X \biimp \varphi(n)),
\end{align*}
where $\varphi$ is $\Sigma^0_1$, $\psi$ is $\Pi^0_1$, and $X$ is not free in $\varphi$. 
\end{itemize}

`$\rca$' stands for `recursive comprehension axiom,' which refers to the $\Delta^0_1$ comprehension scheme, and the `$0$' indicates that the induction scheme is restricted to $\Sigma^0_1$ formulas.  The intuition is that $\rca$ corresponds to computable mathematics.  To show that some set exists when working in $\rca$, one must show how to compute that set from existing sets.

$\rca$ proves many variants of the $\Sigma^0_1$ induction scheme, which we list here for the reader's reference.  In the list below, $\varphi$ is a formula and $\Gamma$ is a class of formulas.
\begin{itemize}
\item The \emph{induction axiom for $\varphi$} is the universal closure of \eqref{fmla-IND} above.  The $\Gamma$ induction scheme consists of the induction axioms for all $\varphi \in \Gamma$.

\medskip

\item The \emph{least element principle for $\varphi$} is the universal closure of the formula
\begin{align*}
\exists n \varphi(n) \imp \exists n[\varphi(n) \andd (\forall m < n)(\neg\varphi(m))].
\end{align*}
The $\Gamma$ least element principle consists of the least element principles for all $\varphi \in \Gamma$.

\medskip

\item The \emph{bounded comprehension axiom for $\varphi$} is the universal closure of the formula
\begin{align*}
\forall b \exists X \forall n[n \in X \biimp (n < b \andd \varphi(n))],
\end{align*}
where $X$ is not free in $\varphi$.  The bounded $\Gamma$ comprehension scheme consists of the bounded comprehension axioms for all $\varphi \in \Gamma$.

\medskip

\item The \emph{bounding (or collection) axiom for $\varphi$} is the universal closure of the formula
\begin{align*}
\forall a[(\forall n < a)(\exists m)\varphi(n,m) \imp \exists b(\forall n < a)(\exists m < b)\varphi(n,m)],
\end{align*}
where $a$ and $b$ are not free in $\varphi$.  The $\Gamma$ bounding scheme consists of the bounding axioms for all $\varphi \in \Gamma$.
\end{itemize}

In addition to $\iso$, $\rca$ proves
\begin{itemize}
\item the $\Pi^0_1$ induction scheme ($\ipo$);

\medskip

\item the $\Sigma^0_1$ least element principle and the $\Pi^0_1$ least element principle;

\medskip

\item the bounded $\Sigma^0_1$ comprehension scheme and the bounded $\Pi^0_1$ comprehension scheme;

\medskip

\item the $\Sigma^0_1$ bounding scheme ($\bso$).
\end{itemize}

The schemes $\iso$, $\ipo$, the $\Sigma^0_1$ least element principle, the $\Pi^0_1$ least element principle, the bounded $\Sigma^0_1$ comprehension scheme, and the bounded $\Pi^0_1$ comprehension scheme are all equivalent over $\pam$ (or over $\pam$ plus $\Delta^0_1$ comprehension in the case of the bounded comprehension schemes).  The scheme $\bso$ is weaker.  $\rca$ does \textbf{not} prove the $\Pi^0_1$ bounding scheme ($\bpo$), which is equivalent to both the $\Sigma^0_2$ bounding scheme ($\bst$) and the $\Delta^0_2$ induction scheme.  We refer the reader to~\cite[Section~I.2]{HajekPudlak} and \cite[Section~II.3]{SimpsonSOSOA} for proofs of these facts.  The equivalence of $\bst$ and $\Delta^0_2$ induction is proved in~\cite{SlamanInd}.

$\rca$ suffices to implement the typical codings ubiquitous in computability theory.  Finite strings, finite sets, integers, rational numbers, etc.\ are coded in the usual way.  Real numbers are coded by rapidly converging Cauchy sequences of rational numbers.  $\rca$ also suffices to define Turing reducibility $\leqT$ and an effective sequence $(\Phi_e)_{e \in \Nb}$ of all Turing functionals (see~\cite[Section~VII.1]{SimpsonSOSOA}).

\subsection*{Notation} Let us fix some notation and terminology for strings.  $\Nb^{<\Nb}$ denotes the set of all finite strings, and $2^{<\Nb}$ denotes the set of all finite binary strings.  We also sometimes use $2^n$ to denote the set of binary strings of length $n$, use $2^{< n}$ to denote the set of binary strings of length less than $n$, etc.  For strings $\sigma$ and $\tau$, $|\sigma|$ denotes the length of $\sigma$, $\sigma \subseteq \tau$ denotes that $\sigma$ is an initial segment of $\tau$, $\sigma^\smf\tau$ denotes the concatenation of $\sigma$ and $\tau$, and $\sigma \rst n = \la \sigma(0), \dots, \sigma(n-1) \ra$ denotes the initial segment of $\sigma$ of length $n$ (when $n \leq |\sigma|$).  The `$\subseteq$' and `$\rst$' notation extend to second-order objects, thought of as infinite strings.  For example, $\sigma \subseteq X$ denotes that $\sigma$ is an initial segment of $X$, and $X \rst n = \la X(0), \dots, X(n-1) \ra$ denotes the initial segment of $X$ of length $n$.  For a string $\sigma$, $[\sigma]$ denotes the basic open set determined by $\sigma$, i.e., the class of all $X$ such that $\sigma \subseteq X$.  Likewise, if $U$ is a set of strings, then $[U]$ represents the open set determined by $U$, and $X \in [U]$ abbreviates $(\exists \sigma \in U)(\sigma \subseteq X)$.  As usual, a \emph{tree} is a set $T \subseteq \Nb^{<\Nb}$ that is closed under initial segments:  $\forall \sigma \forall \tau((\sigma \in T \andd \tau \subseteq \sigma) \imp \tau \in T)$.  $T^n = \{\sigma \in T : |\sigma| = n\}$ denotes the $n$\textsuperscript{th} level of tree $T$.  A function $f$ is a \emph{path} through a tree $T$ if every initial segment of $f$ is in $T$:  $\forall n (f \rst n \in T)$.  $[T]$ denotes the set of paths through tree $T$.

We follow the common convention distinguishing the two symbols `$\Nb$' and `$\omega$' in reverse mathematics.  `$\Nb$' denotes the (possibly non-standard) first-order part of whatever structure is implicitly under consideration, whereas `$\omega$' denotes the standard natural numbers.  We write `$\Nb$' when explicitly working in a formal system, such as when proving some implication over $\rca$.  We write `$\omega$' when constructing a standard model of $\rca$ witnessing some non-implication.

\section{Formalizing algorithmic randomness in second-order arithmetic}\label{sec-RandFormal}

Here and at the beginning of Section~\ref{sec-Implications} we provide a reference for  formalized definitions from effective topology and algorithmic randomness for use in $\rca$, following the style of~\cite{AvigadDeanRute}.   The notions we review here easily form  a linear hierarchy according to randomness strength; however, it will require some effort to verify these implications  in $\rca$.  

In order to define Martin-L\"of randomness in $\rca$, we must define (codes for) effectively open sets and the measures of these sets.  We could of course consider $2^\Nb$ as a complete separable metric space in $\rca$ (see~\cite[Section~II.5]{SimpsonSOSOA}) and use the corresponding notion of open set.  Instead, we use the following equivalent formulation because it more closely resembles the definition used in algorithmic randomness, and it makes defining an open set's measure a little easier.
\begin{Definition}[$\rca$]\label{def-EffOpenSet}{\ }
\begin{itemize}
\item A \emph{code for a $\mbf{\Sigma^0_1}$ set} (or an \emph{open} set) is a sequence $(B_i)_{i \in \Nb}$, where each $B_i$ is a coded finite subset of $2^{<\Nb}$.

\medskip

\item A \emph{code for a $\Sigma^{0,Z}_1$ set} is a code $(B_i)_{i \in \Nb}$ for a $\mbf{\Sigma^0_1}$ set with $(B_i)_{i \in \Nb} \leqT Z$.

\medskip

\item A \emph{code for a uniform sequence of $\Sigma^{0,Z}_1$ sets} is a double-sequence $(B_{n,i})_{n,i \in \Nb} \leqT Z$, where $(B_{n,i})_{i \in \Nb}$ is a code for a $\Sigma^{0,Z}_1$ set for each $n \in \Nb$.

\medskip

\item If $\MP U = (B_i)_{i \in \Nb}$ codes a $\mbf{\Sigma^0_1}$ set, then `$X \in \MP U$' denotes $\exists i (X \in [B_i])$.
\end{itemize}
\end{Definition}

Equivalently, we could take a code for a $\Sigma^{0, Z}_1$ set to be an index $e$ for $W_e^Z = \dom(\Phi_e^Z)$.  Typically, we write `$\MP U$ is a $\Sigma^{0, Z}_1$ set' and `$(\MP{U}_n)_{n \in \Nb}$ is a uniform sequence of $\Sigma^{0, Z}_1$ sets' instead of `$\MP U$ codes a $\Sigma^{0, Z}_1$ set' and `$(\MP{U}_n)_{n \in \Nb}$ codes a uniform sequence of $\Sigma^{0, Z}_1$ sets.'

Now we define Lebesgue measure for $\mbf{\Sigma^0_1}$ sets.
\begin{Definition}[$\rca$]{\ }
\begin{itemize}
\item Let $B \subseteq 2^{<\Nb}$ be finite.  Define $\mu(B) = \sum_{\sigma \in \widehat B}2^{-|\sigma|}$, where\\
$\widehat B= \{\sigma \in B : \text{$\sigma$ has no proper initial segment in $B$}\}$.

\medskip

\item Let $\MP U = (B_i)_{i \in \Nb}$ be a $\mbf{\Sigma^0_1}$ set.
\begin{itemize}
\item The \emph{Lebesgue measure} of $\MP U$ is $\mu(\MP U) = \lim_m \mu(\bigcup_{i \leq m}B_i)$ (if the limit exists).
\item For $r \in \Rb$, `$\mu(\MP U) > r$' denotes $\exists m (\mu(\bigcup_{i \leq m}B_i) > r)$.
\item For $r \in \Rb$, `$\mu(\MP U) \leq r$' denotes $\forall m (\mu(\bigcup_{i \leq m}B_i) \leq r)$.
\end{itemize}
\end{itemize}
\end{Definition}

We warn the reader that $\rca$ is not strong enough to prove that the limit defining $\mu(\MP U)$ exists for every $\mbf{\Sigma^0_1}$ set $\MP U$, which is why we must give explicit definitions for $\mu(\MP U) > r$ and $\mu(\MP U) \leq r$.  In $\rca$, the assertion $\mu(\MP U) = r$ includes the implicit assertion that the limit exists.

Now we can define the notions of algorithmic randomness that we consider.  We start with Martin-L\"of randomness.
\begin{Definition}[$\rca$]{\ }
\begin{itemize}
\item A \emph{$\Sigma^{0,Z}_1$-test} (or \emph{Martin-L\"of test relative to $Z$}) is a uniform sequence $(\MP{U}_n)_{n \in \Nb}$ of $\Sigma^{0,Z}_1$ sets such that $\forall n(\mu(\MP{U}_n) \leq 2^{-n})$.

\medskip

\item $X$ is \emph{$1$-random relative to $Z$} (or \emph{Martin-L\"of random relative to $Z$}) if $X \notin \bigcap_{n \in \Nb}\MP{U}_n$ for every $\Sigma^{0,Z}_1$-test $(\MP{U}_n)_{n \in \Nb}$.

\medskip

\item $\mlr$ is the statement ``for every $Z$ there is an $X$ that is $1$-random relative to $Z$.''
\end{itemize}
\end{Definition}

A notion stronger than Martin-L\"of randomness is weak $2$-randomness.  A weak $2$-test generalizes the concept of a Martin-L\"of test in that one no longer bounds the rate at which the measures of the components of the test converge to $0$.

\begin{Definition}[$\rca$]{\ }
\begin{itemize}
\item A \emph{weak $2$-test relative to $Z$} is a uniform sequence $(\MP{U}_n)_{n \in \Nb}$ of $\Sigma^{0,Z}_1$ sets such that $\lim_n \mu(\MP{U}_n) = 0$, meaning that $\forall k \exists n (\forall m > n)(\mu(\MP{U}_m) \leq 2^{-k})$.

\medskip

\item $X$ is \emph{weakly $2$-random relative to $Z$} if $X \notin \bigcap_{n \in \Nb}\MP{U}_n$ for every weak $2$-test $(\MP{U}_n)_{n \in \Nb}$ relative to $Z$.

\medskip

\item $\wtr$ is the statement ``for every $Z$ there is an $X$ that is weakly $2$-random relative to $Z$.''
\end{itemize}
\end{Definition}

Even stronger is $2$-randomness, which we define here  in terms of $\Sigma^{0, Z}_2$-tests.  We must first define $\mbf{\Sigma}^0_2$ sets and their measures.

\begin{Definition}[$\rca$]\label{def-Sigma2Class}{\ }
\begin{itemize}
\item A \emph{code for a $\mbf{\Sigma}^0_2$ set} is a sequence $(T_i)_{i \in \Nb}$ of subtrees of $2^{<\Nb}$.

\medskip

\item A \emph{code for a $\Sigma^{0,Z}_2$ set} is a code for a $\mbf{\Sigma}^0_2$ set $(T_i)_{i \in \Nb}$ with $(T_i)_{i \in \Nb} \leqT Z$.

\medskip

\item A \emph{code for a uniform sequence of $\Sigma^{0,Z}_2$ sets} is a double-sequence $(T_{n,i})_{n,i \in \Nb} \leqT Z$, where $(T_{n,i})_{i \in \Nb}$ is a code for a $\Sigma^{0,Z}_2$ set for each $n \in \Nb$.

\medskip

\item If $\MP W = (T_i)_{i \in \Nb}$ codes a $\mbf{\Sigma}^0_2$ set, then $X \in \MP W$ denotes $\exists i \forall n (X \rst n \in T_i)$.
\end{itemize}
\end{Definition}

Again, we write `$\MP W$ is a $\Sigma^{0,Z}_2$ set,' etc.\ instead of `$\MP W$ codes a $\Sigma^0_2$ set,' etc. 

\begin{Definition}[$\rca$]
Let $(T_i)_{i \in \Nb}$ be a sequence of trees that codes the $\mbf{\Sigma}^0_2$ set $\MP W$.  Let $q \in \Qb$.  Then `$\mu(\MP W) \leq q$' denotes $\forall i \exists n(2^{-n}|\bigcup_{j \leq i}T_j^n| \leq q)$.
\end{Definition}

\begin{Definition}[$\rca$; \cite{AvigadDeanRute}]{\ }
\begin{itemize}
\item  A \emph{$\Sigma^{0,Z}_2$-test} is a uniform sequence $(\MP{W}_n)_{n \in \Nb}$ of $\Sigma^{0,Z}_2$ sets such that $\forall n(\mu(\MP{W}_n) \leq 2^{-n})$.

\medskip

\item A set  $X$ is \emph{$2$-random relative to $Z$} if $X \notin \bigcap_{n \in \Nb}\MP{W}_n$ for every $\Sigma^{0,Z}_2$-test $(\MP{W}_n)_{n \in \Nb}$.

\medskip

\item $\mlrr{2}$ is the statement ``for every $Z$ there is an $X$ that is $2$-random relative to $Z$.''
\end{itemize}
\end{Definition}

\section{$\mlrr{2}$ and $C$-incompressibility over $\rca$}\label{sec-Inc}
  The statement $\mlrr 2$  (i.e., the existence of 2-random sets) is well-studied  in reverse mathematics.  For instance, in the presence of  the scheme $\bst$  (i.e., $\Sigma^0_2$-bounding), $\mlrr 2$ is equivalent to two formalizations of the dominated convergence theorem~\cite{AvigadDeanRute}, and it implies the rainbow Ramsey theorem for pairs and $2$-bounded colorings~\cites{CsimaMileti,ConidisSlaman}.
 
The goal of this section is to prove the equivalence between $2$-randomness and infinitely often $C$-incompressibility in $\rca$.  The difficulty in doing so is in avoiding arbitrary computations relative to $Z'$ for a set $Z$ (in the sense described the discussion of $\dnr$ in Section~\ref{sec-NonImp}).  In general, $\bst$ is required to show that if $\forall n(\Phi^{Z'}(n)\da)$, then for every $n$ the sequence $\sigma = \la \Phi^{Z'}(0), \dots, \Phi^{Z'}(n-1) \ra$ of the first $n$ values of $\Phi^{Z'}$ exists because this is essentially an arbitrary instance of bounded $\Delta^0_2$ comprehension, which is equivalent to $\Delta^0_2$ induction and hence to $\bst$.  Thus there is a danger of needing $\bst$ when working with computations relative to $Z'$ in $\rca$.  Furthermore, we wish to give proofs that are as concrete as possible, meaning that we prefer to work with objects that exists as sets in $\rca$, such as codes for tests, rather than with virtual objects defined by formulas, such as $Z'$ and sets computable from $Z'$.  This is one reason why we prefer the formalization of $2$-randomness relative to $Z$ in terms of $\Sigma^{0,Z}_2$-tests to the formalization in terms of $1$-randomness relative to $Z'$.

In $\rca$, we may define the standard optimal plain oracle machine $\Vb$ from an effective sequence of all Turing functionals in the usual way.  We may then discuss plain complexity relative to a set $Z$ by writing
\begin{itemize}
\item $C^Z(\sigma) \leq n$ if there is a $\tau$ such that $|\tau| \leq n$ and $\Vb^Z(\tau) = \sigma$ (and similarly with `$<$' in place of `$\leq$');

\medskip

\item $C^Z(\sigma) > n$ if $C^Z(\sigma) \nleq n$ (and similarly with `$\geq$' in place of `$>$');

\medskip

\item $C^Z(\sigma) = n$ if $n$ is least such that $C^Z(\sigma) \leq n$.
\end{itemize}

 $\rca$ proves, using $\iso$ in the guise of the $\Sigma^0_1$ least element principle, that for every $\sigma$ there is an $n$ such that $C(\sigma) = n$.  However, the function $\sigma \mapsto C(\sigma)$ is not computable and therefore   $\rca$ does not prove that this function exists.

\begin{Definition}[$\rca$]{\ }
\begin{itemize}
\item $X$ is \emph{eventually $C^Z$-compressible} if $\forall b \forall^\infty m (C^Z(X \rst m) < m-b)$.

\medskip

\item $X$ is \emph{infinitely often $C^Z$-incompressible} if $\exists b \exists^\infty m(C^Z(X \rst m) \geq m-b)$.

\medskip

\item $\cinc$ is the statement ``for every $Z$ there is an $X$ that is infinitely often $C^Z$-incompressible.''
\end{itemize}
\end{Definition}

We first show that $\rca \vdash \cinc \imp \mlrr{2}$.  The original proof that every infinitely often $C$-incompressible set is $2$-random makes use of prefix-free Kolmogorov complexity relative to $0'$, which we wish to avoid.  We give a proof that is similar to the one given in~\cite{BMSV}.  To do this, we use the following parameterized version of~\cite[Proposition~2.1.14]{NiesBook}, which says that if $\rho(p,n,\tau,Z)$ defines a sequence of $\Sigma^{0,Z}_1$ `sets' (`sets' in quotation because, in $\rca$, $\rho$ may not literally define a set) of requests indexed by $p$, then there is a machine $M$ such that, for every $p$, $M(p, \cdot)$ honors request set $p$.

\begin{Proposition}[$\rca$]\label{prop-MachineExist}
Let $Z$ be a set and suppose that $\rho(p,n,\tau,Z)$ is a $\Sigma^0_1$ formula such that, for each $p, n \in \Nb$, there are at most $2^n$ strings $\tau \in 2^{<\Nb}$ such that $\rho(p,n,\tau,Z)$ holds.  Then there is a machine $M$ such that
\begin{align*}
(\forall p,n \in \Nb)(\forall \tau \in 2^{<\Nb})[\rho(p,n,\tau,Z) \biimp (\exists \sigma \in 2^n)(M^Z(p,\sigma) = \tau)].
\end{align*}
\end{Proposition}

\begin{proof}
The proof is a straightforward (even in $\rca$) extension of the proof of~\cite[Proposition~2.1.14]{NiesBook}.
\end{proof}

\begin{Theorem}
\begin{align*}
\rca \vdash \forall X \forall Z(\text{$X$ is infinitely often $C^Z$-incompressible} \imp \text{$X$ is $2$-random relative to $Z$}).
\end{align*}
Hence $\rca \vdash \cinc \imp \mlrr{2}$.
\end{Theorem}

\begin{proof}
We work in $\rca$ and show that for every $X$ and $Z$, if $X$ is not $2$-random relative to $Z$, then $X$ is eventually $C^Z$-compressible.

Suppose $X$ and $Z$ are sets where $X$ is not $2$-random relative to $Z$.  Let $(T_{n,i})_{n,i \in \Nb} \leqT Z$ be a code for a $\Sigma^{0, Z}_2$-test $(\MP{U}_n)_{n \in \Nb}$ capturing $X$.  Assume that $(\forall n,i,j)(i \leq j \imp T_{n,i} \subseteq T_{n,j})$ by replacing each $T_{n,j}$ by $\bigcup_{i \leq j}T_{n,i}$ if necessary.  Note that $(\forall n,i)(\mu([T_{n,i}]) \leq 2^{-n})$ because $(\MP{U}_n)_{n \in \Nb}$ is a $\Sigma^{0, Z}_2$-test.

Recall that for a tree $T$, $T^m = \{\sigma \in T : |\sigma| = m\}$ denotes the $m$\textsuperscript{th} level of $T$.  To compress the initial segments of $X$, define a parameterized $\Sigma^{0, Z}_1$ set of requests as follows.  First, uniformly define auxiliary sequences $p < m_{p,0} < m_{p,1} < m_{p,2} < \cdots$ for each $p \in \Nb$ so that $(\forall p,i)(|T_{p+1,i}^{m_{p,i}}| \leq 2^{m_{p,i}-p})$, which is possible because $(\forall p,i)(\mu([T_{p+1,i}]) \leq 2^{-(p+1)})$.   Let
\begin{align*}
\rho(p, n, \tau, Z) = \exists i[(\tau \in T_{p+1,i}^{p+n}) \andd (m_{p,i} \leq p+n < m_{p,i+1})].
\end{align*}
If $\rho(p, n, \tau, Z)$ holds, then it must be that $\tau \in T_{p+1,i}^{p+n}$ for the $i$ such that $m_{p,i} \leq p+n < m_{p,i+1}$.  There are at most $2^{m_{p,i} - p} \leq 2^n$ such $\tau$ by the choice of $m_{p,i}$.  Thus for every $p,n \in \Nb$, there are at most $2^n$ strings $\tau \in 2^{<\Nb}$ such that $\rho(p,n,\tau,Z)$ holds.  Thus let $M$ be as in the conclusion of Proposition~\ref{prop-MachineExist} for this $\rho$.  Let $N$ be a machine such that $(\forall p \in \Nb)(\forall \sigma \in 2^{<\Nb})(N^Z({0^p}^\smf{1}^\smf\sigma) = M^Z(2^p, \sigma))$ (here we warn the reader that in $N$, `$0^p$' is the string of $0$'s of length $p$, but in $M$, `$2^p$' is the number $2^p$).  Let $c \in \Nb$ be a constant such that $\forall \tau(C^Z(\tau) \leq C_N^Z(\tau) + c)$.

We show that $\forall b \forall^\infty m(C^Z(X \rst m) < m-b)$ by showing that $\forall b \forall^\infty m(C_N^Z(X \rst m) < m-b-c)$.  Fix $b \in \Nb$.  Let $p$ be large enough so that $2^p > b+c+p+1$.  By the fact that $(\MP{U}_n)_{n \in \Nb}$ captures $X$, let $i_0$ be such that $(\forall i \geq i_0)(X \in [T_{2^p+1, i}])$.  Now consider any $n \geq  m_{2^p, i_0} - 2^p$.  Let $i \geq i_0$ be the $i$ such that $m_{2^p, i} \leq 2^p + n < m_{2^p, i+1}$.  Then $X \rst (2^p+n) \in T_{2^p+1, i}^{2^p + n}$ by the choice of $i_0$, so $\rho(2^p, n, X \rst (2^p + n), Z)$.  Thus there is a $\sigma \in 2^n$ such that $N^Z({0^p}^\smf{1}^\smf\sigma) = M^Z(2^p, \sigma) = X \rst (2^p + n)$.  Thus 
\begin{align*}
C_N^Z(X \rst (2^p + n)) \leq p + 1 + |\sigma| = p+1+n < 2^p + n - b - c.
\end{align*}
Therefore, if $m \geq m_{2^p, i_0}$, then $C_N^Z(X \rst m) < m-b-c$, as desired.  Thus $X$ is eventually $C^Z$-compressible.
\end{proof}

Next  we show the harder implication that $\rca \vdash \mlrr{2} \imp \cinc$.  The proof in Miller~\cite{MillerKRand} that every $2$-random set is infinitely often $C$-incompressible uses the familiar characterization of $2$-random sets in terms of prefix-free Kolmogorov complexity relative to $0'$, which we wish to avoid.  The proof in Nies, Stephan, and Terwijn~\cite{NiesStephanTerwijn} (see also Nies~\cite[Theorem 3.6.10]{NiesBook}) uses the so-called \emph{compression functions} and an application of the low basis theorem.  Though we did not pursue this approach in detail, we believe that it is possible to give a metamathematical version of the argument via compression functions in $\rca$ by following the proof of~\cite[Theorem 3.6.10]{NiesBook} and using a carefully formalized version of the low basis theorem, such as H\'{a}jek and Pudlak~\cite[Theorem~I.3.8]{HajekPudlak}.  This strategy could be implemented entirely (and quite delicately) in $\rca$, or it could be implemented by observing that the desired statement
\begin{align*}
\forall X \forall Z (\text{$X$ is $2$-random relative to $Z$} \imp \text{$X$ is infinitely often $C^Z$-incompressible})\label{fmla-2RandImpCinc}\tag{$\star$}
\end{align*}
is $\Pi^1_1$ and by appealing to conservativity.  A classic result of Harrington is that every countable model of $\rca$ can be extended to a countable model of $\wkl$ with the same first-order part (see~\cite[Theorem~IX.2.1]{SimpsonSOSOA}).  It follows that $\wkl$ is $\Pi^1_1$-conservative over $\rca$.  By combining the proof of Harrington's result with the proof of the formalized low basis theorem from H\'{a}jek and Pudlak, one may ensure that the sets in the extended model of $\wkl$ are all low in the sense of H\'{a}jek and Pudlak.  This yields that $\rca$ plus the statement ``every infinite binary-branching tree has a low infinite path'' is $\Pi^1_1$-conservative over $\rca$.  The conceptual advantage of the conservativity strategy over the directly-in-$\rca$ strategy is that one may assume that the desired compression function actually exists as a second-order object instead of merely being defined by some formula.  We thank Keita Yokoyama for many helpful comments concerning metamathematical approaches to showing that $\rca \vdash$~\eqref{fmla-2RandImpCinc}.

We prefer a concrete argument given in $\rca$ to the metamathematical approach outlined above, and   find it interesting that a concrete argument is possible.  Our argument is a formalization of the proof presented in Bauwens~\cite{Bauwens}, which itself is based on the proof in Bienvenu et al.~\cite{BMSV}.  The proof in~\cite{Bauwens} proceeds via the following  covering result.

\begin{Theorem}[{Conidis \cite[Theorem~3.1]{Conidis}}]\label{thm-Conidis}
Let $q \in \Qb$, and let $(\MP{U}_i)_{i \in \omega}$ be a uniform sequence of $\Sigma^0_1$ sets such that $  \mu(\MP{U}_i) \leq q$ for each $i$.  For every $p \in \Qb$ with $p > q$, there is a $\Sigma^{0,0'}_1$ set $\MP V$ such that $\mu(\MP V) \leq p$ and $ \bigcap_{i \geq N}\MP{U}_i \subseteq \MP V$ for each $N$.  Furthermore, $\MP V$ is produced uniformly from an index $e$ such that $\Phi_e = (\MP{U}_i)_{i \in \Nb}$ as well as $q$ and $p$.
\end{Theorem}

Assuming  Theorem~\ref{thm-Conidis}, we sketch the argument that no eventually $C$-compressible set $X$ is $2$-random.  Suppose that $\forall^\infty i (C(X \rst i) < i-b)$ for each $b$.  We want to find a $\Sigma^{0,0'}_1$-test capturing $X$.  Define a double-sequence $(\MP{U}_{b,i})_{b,i \in \omega}$ of $\Sigma^0_1$ sets by taking $\MP{U}_{b,i} = \{Y : C(Y \rst i) < i-b\}$.  Then $ \mu(\MP{U}_{b,i}) \leq 2^{-b}$ for each $b$ and $i$.  By Theorem~\ref{thm-Conidis}, we obtain a $\Sigma^{0,0'}_1$-test $(\MP{V}_b)_{b \in \omega}$ such that $ \bigcap_{i \geq N}\MP{U}_{b+1, i} \subseteq \MP{V}_b$ for each $b$ and $N$.  The test $(\MP{V}_b)_{b \in \omega}$ captures $X$ because for each $b$ there is an $N$ such that $(\forall i > N)(C(X \rst i) < i-(b+1))$, and hence   $  X \in \bigcap_{i \geq N}\MP{U}_{b+1, i}$, which is contained in $ \MP{V}_b$.  Thus $X$ is not $2$-random.

  The proof of Theorem~\ref{thm-Conidis} in~\cites{Bauwens}  makes use of an  inclusion-exclusion principle for open sets provable in $\rca$.  We include the standard proof to in order convince the reader that it can be carried out in $\rca$.
\begin{Lemma}[$\rca$]\label{lem-IncExcl}
Let $\MP A, \MP B \subseteq 2^{\Nb}$ be open sets, and let $a,b,r \in \Qb^{\geq 0}$ be such that $\mu(\MP A) \leq a$, $\mu(\MP B) \leq b$ and $\mu(\MP{A} \cup \MP{B}) > r$.  Then $\mu(\MP{A} \cap \MP{B}) \leq a+b-r$.
\end{Lemma}

\begin{proof}
Suppose for a contradiction that $\mu(\MP{A} \cap \MP{B}) > a+b-r$.  Then $\mu(\MP{A} \cap \MP{B}) > a+b-r + 2^{-n}$ for some $n \in \Nb$, so there is a clopen $\MP C \subseteq \MP{A} \cap \MP{B}$ with $a+b-r + 2^{-(n+1)} \leq \mu(\MP C) \leq a+b-r + 2^{-n}$.  Let $\MP{A}_0 = \MP{A} \setminus \MP{C}$, and let $\MP{B}_0 = \MP{B} \setminus \MP{C}$.  Note that $\mu(\MP{A}_0) \leq a - (a+b-r + 2^{-(n+1)}) = r-b - 2^{-(n+1)}$ and that $\mu(\MP{B}_0) \leq b - (a+b-r + 2^{-(n+1)}) = r-a - 2^{-(n+1)}$.  Then
\begin{align*}
\mu(\MP{A} \cup \MP{B}) &\leq \mu(\MP{A}_0) + \mu(\MP{B}_0) + \mu(\MP C)\\
&\leq (r-b - 2^{-(n+1)}) + (r-a - 2^{-(n+1)}) + (a+b-r + 2^{-n}) = r. 
\end{align*}
This contradicts $\mu(\MP{A} \cup \MP{B}) > r$.
\end{proof}

Lemma~\ref{lem-2MLRimpINCHelper} formalizes Theorem~\ref{thm-Conidis} for use in $\rca$.  Notice that  the set  $\MP V$ produced is now a $\Sigma^{0,Z}_2$ set, rather than  a $\Sigma^{0,Z'}_1$ set.

\begin{Lemma}[$\rca$]\label{lem-2MLRimpINCHelper}
Let $Z$ be a set, let $q \in \Qb$, and let $(\MP{U}_i)_{i \in \Nb}$ be a uniform sequence of $\Sigma^{0,Z}_1$ sets such that $\forall i(\mu(\MP{U}_i) \leq q)$.  Then, for every $p \in \Qb$ with $p > q$, there is a $\Sigma^{0,Z}_2$ set $\MP V$ such that $\mu(\MP V) \leq p$ and $\forall N(\bigcap_{i \geq N}\MP{U}_i \subseteq \MP V)$.  Furthermore, $\MP V$ is produced uniformly from $Z$, an index $e$ such that $\Phi_e^Z = (\MP{U}_i)_{i \in \Nb}$, $q$, and $p$.
\end{Lemma}

\begin{proof}
The basic idea is to replace
\begin{align*}
\bigcup_{N \in \Nb} \bigcap_{i \geq N} \MP{U}_i
\end{align*}
by a superset of the form
\begin{align*}
\MP V =  \bigcup_{N \in \Nb} \bigcap_{i = N}^{b_N} \MP{U}_i
\end{align*}
 for an appropriate sequence $0 < b_0 < b_1 < \cdots$ because $\bigcup_{N \in \Nb} \bigcap_{i \geq N} \MP{U}_i$ is too complicated, whereas $\bigcup_{N \in \Nb} \bigcap_{i = N}^{b_N} \MP{U}_i$ is open (but in our case not effectively open; we produce a $\Sigma^0_2$ code for a set that happens to be open).

We want to identify a sequence $0 < b_0 < b_1 < \cdots$ that yields $\mu(\MP V) \leq p$.     The proof in~\cite{Bauwens} computes such a sequence from $0'$.  We wish to avoid explicit computations relative to $0'$ because the analysis of such computations has the danger of possibly requiring $\bst$.

First some notation.  For $a, b \in \Nb$ with $a < b$, let $\MP{U}_{a \ldots b} = \bigcap_{i=a}^b \MP{U}_i$.  For a sequence $\la b_0, b_1, \dots, b_{n-1} \ra$ with $0 < b_0 < b_1 < \cdots < b_{n-1}$, let $$\MP{S}_{\la b_0, \dots, b_{n-1} \ra} = \bigcup_{j < n} \MP{U}_{j \ldots b_j}.$$  We can fix codes for these sets:
\begin{itemize}
\item Let $(U_{i,s})_{i,s \in \Nb} \leqT Z$ denote the code for $(\MP{U}_i)_{i \in \Nb}$ so that, for all $i$, $\bigcup_{s \in \Nb}[U_{i,s}] = \MP{U}_i$.

\medskip

\item From $(U_{i,s})_{i,s \in \Nb}$, define codes $(U_{a \ldots b, s})_{s \in \Nb} \leqT Z$ uniformly for all $a,b \in \Nb$ with $a < b$ so that $\bigcup_{s \in \Nb}[U_{a \ldots b, s}] = \MP{U}_{a \ldots b}$.

\medskip

\item Similarly, for every sequence $\la b_0, b_1, \dots, b_{n-1} \ra$ with $0 < b_0 < b_1 < \cdots < b_{n-1}$, uniformly define codes $(S_{\la b_0, \dots, b_{n-1} \ra, s})_{s \in \Nb} \leqT Z$ so that $\bigcup_{s \in \Nb}[S_{\la b_0, \dots, b_{n-1} \ra, s}] = \MP{S}_{\la b_0, \dots, b_{n-1} \ra}$.
\end{itemize}

Notice that if $\la b_0, b_1, \dots, b_{n-1} \ra$ is a sequence with $0 < b_0 < b_1 < \cdots < b_{n-1}$, then $\bigcap_{i \geq N}\MP{U}_i \subseteq \MP{S}_{\la b_0, \dots, b_{n-1} \ra}$ holds when $N < n$.  

We would like to define $\MP V$ by taking the union of sets of the form $\MP{S}_{\la b_0, \dots, b_{n-1} \ra}$ for longer and longer sequences $\la b_0, \dots, b_{n-1} \ra$.  However, we also need to ensure that $\mu(\MP V) \leq p$.  Thus we need to find sequences $\la b_0, \dots, b_{n-1} \ra$ where $\MP{S}_{\la b_0, \dots, b_{n-1} \ra}$ has small measure and that  additionally are extendible to longer sequences $\la b_0, \dots, b_{m-1} \ra \supseteq \la b_0, \dots, b_{n-1} \ra$ where $\MP{S}_{\la b_0, \dots, b_{m-1} \ra}$ also has small measure.    Part~(ii) of the following claim says that this is possible: there are sequences $\la b_0, \dots, b_{n-1} \ra$ of arbitrary length such that for every subsequence $\la b_0, \dots, b_k \ra$ with  $k < n$, the set $\MP{S}_{\la b_0, \dots, b_k \ra} \cup \MP{U}_i$ has small measure for all $i > b_k$. The main technical work to prove the claim is in its Part~(i).

\begin{Claim}\label{claim-LongSeq}{\ }
\begin{itemize}
\item[(i)]  For every $a \in \Nb$ and every $r \in \Qb$ with $r > q$, there is  $b > a$ such that $\mu(\MP{U}_{a \ldots b} \cup \MP{U}_i) \leq r$ for each $i>b$.

\medskip

\item[(ii)]  For every $n \in \Nb$ and every $q_0, \dots, q_{n-1} \in \Qb$ with $q < q_0 < \cdots < q_{n-1}$, there is a sequence $\la b_0, b_1, \dots, b_{n-1} \ra$ with $0 < b_0 < b_1 < \cdots < b_{n-1}$ such that $\mu(\MP{S}_{\la b_0, \dots, b_k \ra} \cup \MP{U}_i) \leq q_k)$ for each $ k < n$ and each $  i > b_k$.
\end{itemize}
\end{Claim}

\begin{proof}[Proof of Claim]
(i) Suppose for a contradiction that $(\forall b > a)(\exists i > b)(\mu(\MP{U}_{a \ldots b} \cup \MP{U}_i) > r)$.  Consider for a moment any $b > a$ and a $c > b$ such that $\mu(\MP{U}_{a \ldots b} \cup \MP{U}_c) > r$.  We assume that $\mu(\MP{U}_c) \leq q$, so if $\mu(\MP{U}_{a \ldots b}) \leq x$ for some $x \in \Qb$, then $\mu(\MP{U}_{a \ldots c}) \leq \mu(\MP{U}_{a \ldots b} \cap \MP{U}_c) \leq x - (r-q)$ by Lemma~\ref{lem-2MLRimpINCHelper}.  By iterating this argument sufficiently many times, we find a contradictory $c$ such that $\mu(\MP{U}_{a \ldots c}) < 0$.

To implement this argument formally, consider the formula
\begin{align*}
\varphi(k) = (\exists \la b_0, \dots, b_k \ra \in \Nb)\left[(a < b_0) \andd (\forall i < k)(b_i < b_{i+1}) \andd (\forall i < k)(\mu(\MP{U}_{a \ldots b_i} \cup \MP{U}_{b_{i+1}}) > r)\right].
\end{align*}
The formula $\varphi(k)$ is $\Sigma^{0, Z}_1$ because `$\mu(\MP{U}_{a \ldots b_i} \cup \MP{U}_{b_{i+1}}) > r$' is $\Sigma^{0,Z}_1$.  Thus we may conclude $\forall k \varphi(k)$ by $\iso$ and the assumption $(\forall b > a)(\exists i > b)(\mu(\MP{U}_{a \ldots b} \cup \MP{U}_i) > r)$.  Now choose $k > q/(r-q)$ and, by $\varphi(k)$, let $a < b_0 < b_1 < \cdots < b_k$ be such that $(\forall i < k)(\mu(\MP{U}_{a \ldots b_i} \cup \MP{U}_{b_{i+1}}) > r)$.  Then, for any $x \in \Qb$ and $i < k$, if $\mu(\MP{U}_{a \ldots b_i}) \leq x$, then $\mu(\MP{U}_{a \ldots b_{i+1}}) \leq x - (r-q)$ by Lemma~\ref{lem-IncExcl} and the assumption $\mu(\MP{U}_{b_{i+1}}) \leq q$.  By $\ipo$, we can then conclude that $(\forall i \leq k)[\mu(\MP{U}_{a \ldots b_i}) \leq q - i(r-q)]$.  This is a contradiction because for $i = k$ it gives
\begin{align*}
\mu(\MP{U}_{a \ldots b_k}) \leq q - k(r-q) < q - q = 0.
\end{align*}

\noindent (ii)   Given $n$ and $q_0, \dots, q_{n-1} \in \Qb$ with $q < q_0 < \cdots < q_{n-1}$, let $b > n$ be such that $ \mu(\MP{U}_{n \ldots b} \cup \MP{U}_i) \leq q_0$ for each $i> b$.  Let $b_j = b+j$ for each $j < n$.  Then $(\forall k < n)(\MP{S}_{\la b_0, b_1, \dots, b_k \ra} \subseteq \MP{U}_{n \ldots b})$, so if $i > b_k \geq b$, then $\mu(\MP{S}_{\la b_0, b_1, \dots, b_k \ra} \cup \MP{U}_i) \leq \mu(\MP{U}_{n \ldots b} \cup \MP{U}_i) \leq q_0 \leq q_k$.
\end{proof}

Choose an increasing sequence of rationals $q_0 < q_1 < q_2 < \cdots$ inside the interval $(q, p)$.  We first illustrate some of the ideas behind constructing the code for $\MP V$ before diving into its full construction.  Claim~\ref{claim-LongSeq} part~(ii) tells us that it is possible to find arbitrary long sequences $b_0 < \cdots < b_{n-1}$ with $\mu(\MP{S}_{\la b_0, \dots, b_{n-1} \ra})$ under control that can be extended to even longer sequences with the corresponding measure still under control.  The conclusion of Claim~\ref{claim-LongSeq} part~(ii) is $\Pi^{0, Z}_1$, so we can use trees to identify sequences $b_0 < \cdots < b_{n-1}$ satisfying the conclusion for $q_0, \dots, q_{n-1}$ in the following way.  For each $t$ and $\la b_0, \dots, b_{n-1} \ra$ we can define a tree $T_{\la t, b_0, \dots, b_{n-1} \ra}$ such that
\begin{align*}
[T_{\la t, b_0, \dots, b_{n-1} \ra}] = 
\begin{cases}
[S_{\la b_0, \dots, b_{n-1} \ra, t}] & \text{if $\la b_0, \dots, b_{n-1} \ra$ satisfies Claim~\ref{claim-LongSeq} part~(ii)}\\
\emptyset & \text{otherwise}.
\end{cases}
\end{align*}
This is accomplished by adding to $T_{\la t, b_0, \dots, b_{n-1} \ra}$ all strings comparable with the finitely many strings in $S_{\la b_0, \dots, b_{n-1} \ra, t}$ until possibly noticing that $b_0 < \cdots < b_{n-1}$ does \textbf{not} satisfy Claim~\ref{claim-LongSeq} part~(ii) for $q_0, \dots, q_{n-1}$.

For a fixed $b_0 < \cdots < b_{n-1}$, we then have that
\begin{align*}
\bigcup_{t \in \Nb}&[T_{\la t, b_0, \dots, b_{n-1} \ra}] =\\
&\begin{cases}
\bigcup_{t \in \Nb}[S_{\la b_0, \dots, b_{n-1} \ra, t}] = \MP{S}_{\la b_0, \dots, b_{n-1} \ra} & \text{if $\la b_0, \dots, b_{n-1} \ra$ satisfies Claim~\ref{claim-LongSeq} part~(ii)}\\
\emptyset & \text{otherwise}.
\end{cases}
\end{align*}

Therefore, if we take the sequence $(T_i)_{i \in \Nb}$ of all trees $T_{\la t, b_0, \dots, b_{n-1} \ra}$ for all $t$, $n$, and $b_0 < \cdots < b_{n-1}$ as a code for the $\Sigma^0_2$ set $\MP V$, we get that
\begin{align*}
\MP V = \bigcup \{\MP{S}_{\la b_0, \dots, b_{n-1} \ra} : \text{$\la b_0, \dots, b_{n-1} \ra$ satisfies Claim~\ref{claim-LongSeq} part~(ii)}\}.
\end{align*}
In this case, we certainly have
\begin{align*}
\bigcup_{N \in \Nb} \bigcap_{i \geq N} \MP{U}_i \subseteq \MP V,
\end{align*}
but we have done nothing to help keep track of $\mu(\MP V)$.

So instead of having $\MP V$ contain $\MP{S}_{\la b_0, \dots, b_{n-1} \ra}$ for every $b_0 < \dots < b_{n-1}$ that satisfies Claim~\ref{claim-LongSeq} part~(ii), we want $\MP V$ to contain $\MP{S}_{\la b_0, \dots, b_{n-1} \ra}$ for exactly one $b_0 < \dots < b_{n-1}$ satisfying Claim~\ref{claim-LongSeq} part~(ii) for each $n$.  Moreover, if $n > m$, we want $\la b_0, \dots, b_{n-1} \ra$ to extend $\la b_0, \dots, b_{m-1} \ra$ so that $\MP{S}_{\la b_0, \dots, b_{n-1} \ra} \supseteq \MP{S}_{\la b_0, \dots, b_{m-1} \ra}$, which makes the measures of these sets easier to analyze.  To accomplish this and to give the full construction of the code for $\MP V$, we introduce the notion of a \emph{good} sequence.

Call a sequence $\la b_0, s_0, \dots, b_{n-1}, s_{n-1} \ra$ \emph{good} if $\la s_0, \dots, s_{n-1} \ra$ witnesses that $\la b_0, \dots, b_{n-1} \ra$ is the lexicographically least sequence of length $n$ satisfying Claim~\ref{claim-LongSeq} part~(ii) for $q_0, \dots, q_{n-1}$.  More formally, $\la b_0, s_0, \dots, b_{n-1}, s_{n-1} \ra$ is \emph{good} if
\begin{itemize}
\item[(i)] $0 < b_0 < b_1 < \cdots < b_{n-1}$;

\medskip

\item[(ii)] $(\forall k < n)(\forall i > b_k)(\mu(\MP{S}_{\la b_0, \dots, b_k \ra} \cup \MP{U}_i) \leq q_k)$; and

\medskip

\item[(iii)] for all $k < n$, if $b_k > b_{k-1} + 1$ (or if $b_0 > 1$ in the case $k = 0$), then $s_k = \la i, s \ra$ is such that $i > b_k - 1$ and $\mu([S_{\la b_0, \dots, b_{k-1}, b_k - 1 \ra, s}] \cup [U_{i,s}]) > q_k$.
\end{itemize}

Item (iii) says that if $b_k$ is not as small as possible (i.e., if $b_k > b_{k-1}+1$ or if $b_0 > 1$ in the case $k = 0$), then $s_k$ is a pair witnessing that $b_k$ cannot be chosen smaller and still satisfy Claim~\ref{claim-LongSeq} part~(ii).  It is in this sense that $\la s_0, \dots, s_{n-1} \ra$ witnesses that $\la b_0, \dots, b_{n-1} \ra$ is the lexicographically least sequence of length $n$ satisfying Claim~\ref{claim-LongSeq} part~(ii).  Notice that items~(i) and~(iii) are $\Delta^{0, Z}_1$ and that item~(ii) is $\Pi^{0, Z}_1$, so `$\la b_0, s_0, \dots, b_{n-1}, s_{n-1} \ra$ is good' is $\Pi^{0, Z}_1$.  So instead of defining trees $T_{\la t, b_0, \dots, b_{n-1} \ra}$ as above, we will define similar trees $T_{\la t, b_0, s_0, \dots, b_{n-1}, s_{n-1} \ra}$ so that
\begin{align*}
[T_{\la t, b_0, s_0, \dots, b_{n-1}, s_{n-1} \ra}] = 
\begin{cases}
[S_{\la b_0, \dots, b_{n-1} \ra, t}] & \text{if $\la b_0, s_0, \dots, b_{n-1}, s_{n-1} \ra$ is good}\\
\emptyset & \text{otherwise}.
\end{cases}
\end{align*}
However, before we do this, we show that the good sequences do indeed have their intended properties.
Note that if $\la b_0, s_0, \dots, b_{n-1}, s_{n-1} \ra$ is good and $k \leq n$, then $\la b_0, s_0, \dots, b_{k-1}, s_{k-1} \ra$ is also good.  By the following the good sequences   identify a unique infinite sequence $0 < b_0 < b_1 < \cdots$, which is the sequence we use to define $\MP V$.

\begin{Claim}\label{claim-UniqueSeq}
For each $n$ there is exactly one sequence $b_0 < \cdots < b_{n-1}$ for which there are $s_0, \dots, s_{n-1}$ such that $\la b_0, s_0, \dots, b_{n-1}, s_{n-1} \ra$ is good.
\end{Claim}

\begin{proof}[Proof of Claim] Fix $n$.  We first show that there is at most one sequence $b_0 < \cdots < b_{n-1}$ for which there are $s_0, \dots, s_{n-1}$ such that $\la b_0, s_0, \dots, b_{n-1}, s_{n-1} \ra$ is good.  Suppose that $\la b_0, s_0, \dots, b_{n-1}, s_{n-1} \ra$ and $\la b_0', s_0', \dots, b_{n-1}', s_{n-1}' \ra$ are both good and that (for the sake of argument) there is a $k < n$ such that $b_k < b_k'$ and $(\forall j < k)(b_j = b_j')$.  Then $(\forall i > b_k)(\mu(\MP{S}_{\la b_0, \dots, b_k \ra} \cup \MP{U}_i) \leq q_k)$ because $\la b_0, s_0, \dots, b_{n-1}, s_{n-1} \ra$ is good.  However, $\MP{S}_{\la b_0', \dots, b_{k-1}', b_k'-1 \ra} \subseteq \MP{S}_{\la b_0, \dots, b_k \ra}$ because $b_k \leq b_k'-1$ and $(\forall j < k)(b_j = b_j')$.  Therefore $(\forall i > b_k'-1)(\mu(\MP{S}_{\la b_0', \dots, b_{k-1}', b_k'-1 \ra} \cup \MP{U}_i) \leq q_k)$.  Thus there can be no $s_k' = \la i, s \ra$ such that $i > b_k' - 1$ and $\mu([S_{\la b_0', \dots, b_{k-1}', b_k' - 1 \ra, s}] \cup [U_{i,s}]) > q_k$.  Therefore $\la b_0', s_0', \dots, b_{n-1}', s_{n-1}' \ra$ is not good.

Now we show that there is at least one sequence $b_0 < \cdots < b_{n-1}$ for which there are $s_0, \dots, s_{n-1}$ such that $\la b_0, s_0, \dots, b_{n-1}, s_{n-1} \ra$ is good.  By Claim~\ref{claim-LongSeq} part~(ii) and the $\Pi^0_1$ least element principle, there is a least code $\la b_0, \dots, b_{n-1} \ra$ with $0 < b_0 < \cdots < b_{n-1}$ and such that $(\forall k < n)(\forall i > b_k)(\mu(\MP{S}_{\la b_0, \dots, b_k \ra} \cup \MP{U}_i) \leq q_k)$.  As usual, we tacitly assume that the coding of sequences is increasing in every coordinate.  Let $A$ be the set of $k < n$ such that $b_k > b_{k-1}+1$ (or $b_0 > 1$ in the case $k=0$).  Then, by the minimality of $\la b_0, \dots, b_{n-1} \ra$, $(\forall k \in A)(\exists i > b_k-1)(\mu(\MP{S}_{\la b_0, \dots, b_{k-1}, b_k-1 \ra} \cup \MP{U}_i) > q_k)$ and so $(\forall k \in A)(\exists i > b_k-1)(\exists s)(\mu([S_{\la b_0, \dots, b_{k-1}, b_k-1 \ra, s}] \cup [U_{i, s}]) > q_k)$.  Thus, for every $k \in A$ we may choose an $s_k = \la i, s \ra$ such that $i > b_{k-1}$ and $\mu([S_{\la b_0, \dots, b_{k-1}, b_k-1 \ra, s}] \cup [U_{i,s}]) > q_k$.  Then, letting $s_k = 0$ for all $k < n$ that are not in $A$, we see that $\la b_0, s_0, \dots, b_{n-1}, s_{n-1} \ra$ is good.  \end{proof}

We are now ready to define a code $(T_i)_{i \in \Nb}$ for the desired $\Sigma^{0,Z}_2$ set $\MP V$.  The idea is to arrange that $\MP V = \bigcup_{n \in \Nb} \MP{S}_{\la b_0, \dots, b_{n-1} \ra}$, for the sequence $b_0 < b_1 < \cdots$ identified above.

We view each $i$ as a sequence $i = \la t, b_0, s_0, \dots, b_{n-1}, s_{n-1} \ra$ and use the trees $T_{\la t, b_0, s_0, \dots, b_{n-1}, s_{n-1} \ra}$ to ensure that $\MP{S}_{\la b_0, \dots, b_{n-1} \ra} \subseteq \MP V$ when there are $s_0, \dots, s_{n-1}$ such that $\la b_0, s_0, \dots, b_{n-1}, s_{n-1} \ra$ is good.  Thus for every $\la t, b_0, s_0, \dots, b_{n-1}, s_{n-1} \ra \in \Nb$, we define $T_{\la t, b_0, s_0, \dots, b_{n-1}, s_{n-1} \ra}$ so that
\begin{align*}
[T_{\la t, b_0, s_0, \dots, b_{n-1}, s_{n-1} \ra}] = 
\begin{cases}
[S_{\la b_0, \dots, b_{n-1} \ra, t}] & \text{if $\la b_0, s_0, \dots, b_{n-1}, s_{n-1} \ra$ is good}\\
\emptyset & \text{otherwise},
\end{cases}
\end{align*}
as described above.

To define $T_{\la t, b_0, s_0, \dots, b_{n-1}, s_{n-1} \ra}$, first check that $\la b_0, s_0, \dots, b_{n-1}, s_{n-1} \ra$ satisfies items (i) and (iii) in the definition of `good.'  If the check fails, set $T_{\la t, b_0, s_0, \dots, b_{n-1}, s_{n-1} \ra} = \emptyset$.  If the check passes, then add to $T_{\la t, b_0, s_0, \dots, b_{n-1}, s_{n-1} \ra}$ all initial segments of the strings in $S_{\la b_0, \dots, b_{n-1} \ra, t}$, and then add all extensions of all strings in $S_{\la b_0, \dots, b_{n-1} \ra, t}$, level-by-level, until possibly seeing that $\la b_0, s_0, \dots, b_{n-1}, s_{n-1} \ra$ is not good by the failure of item~(ii) in the definition of `good.'  In the end, if $\la b_0, s_0, \dots, b_{n-1}, s_{n-1} \ra$ is good, then $T_{\la t, b_0, s_0, \dots, b_{n-1}, s_{n-1} \ra}$ consists of all strings comparable with some string in $S_{\la b_0, \dots, b_{n-1} \ra, t}$, so $[T_{\la t, b_0, s_0, \dots, b_{n-1}, s_{n-1} \ra}] = [S_{\la b_0, \dots, b_{n-1} \ra, t}]$.  Otherwise, $T_{\la t, b_0, s_0, \dots, b_{n-1}, s_{n-1} \ra}$ is finite, so we have that $[T_{\la t, b_0, s_0, \dots, b_{n-1}, s_{n-1} \ra}] = \emptyset$.

Formally, if $\la b_0, s_0, \dots, b_{n-1}, s_{n-1} \ra$ is not good by the failure of either (i) or (iii), then let $T_{\la t, b_0, s_0, \dots, b_{n-1}, s_{n-1} \ra} = \emptyset$.  Otherwise, let $T_{\la t, b_0, s_0, \dots, b_{n-1}, s_{n-1} \ra}$ be the set of all strings $\tau \in 2^{< \Nb}$ such that either
\begin{itemize}
\item $\tau \subseteq \sigma$ for some $\sigma \in S_{\la b_0, \dots, b_{n-1} \ra, t}$; or
\item $\tau \supseteq \sigma$ for some $\sigma \in S_{\la b_0, \dots, b_{n-1} \ra, t}$ and $(\forall k < n)(\forall i < |\tau|)(i > b_k \imp \mu([S_{\la b_0, \dots, b_k \ra, |\tau|}] \cup [U_{i, |\tau|}]) \leq q_k)$.
\end{itemize}
That is, in this case we add to $T_{\la t, b_0, s_0, \dots, b_{n-1}, s_{n-1} \ra}$ all extensions of strings in $S_{\la b_0, \dots, b_{n-1} \ra, t}$ until possibly reaching a level witnessing that $\la b_0, s_0, \dots, b_{n-1}, s_{n-1} \ra$ is not good by the failure of (ii).

Let $\MP V$ denote the $\Sigma^{0,Z}_2$ set defined by $(T_i)_{i \in \Nb}$ according to Definition~\ref{def-Sigma2Class}.  To show that $\mu(\MP V) \leq p$, we need to show that $\forall m \exists \ell(2^{-\ell}|\bigcup_{i \leq m}T_i^\ell| \leq p)$.  Fix $m \in \Nb$.  We find an $\ell$ large enough so that each string in $\bigcup_{i \leq m}T_i^\ell$ is an extension of some string in $\bigcup_{t \in \Nb}S_{\la \tilde{b}_0, \dots, \tilde{b}_{\tilde{n}-1} \ra, t}$ for a $\la \tilde{b}_0, \dots, \tilde{b}_{\tilde{n}-1} \ra$ for which there are $\tilde{s}_0, \dots, \tilde{s}_{\tilde{n}-1}$ such that $\la \tilde{b}_0, \tilde{s}_0, \dots, \tilde{b}_{\tilde{n}-1}, \tilde{s}_{\tilde{n}-1} \ra$ is good.  Once we have $\ell$, it follows that $2^{-\ell}|\bigcup_{i \leq m}T_i^\ell| \leq p$ because  then
\begin{align*}
2^{-\ell}\left|\bigcup_{i \leq m}T_i^\ell\right| \leq \mu(\MP{S}_{\la \tilde{b}_0, \dots, \tilde{b}_{\tilde{n}-1} \ra}) \leq q_{\tilde{n}-1} < p.
\end{align*}

To find $\ell$, first use bounded $\Pi^0_1$ comprehension to let $A$ be the set of all $\la t, b_0, s_0, \dots, b_{n-1}, s_{n-1} \ra \leq m$ such that $\la b_0, s_0, \dots, b_{n-1}, s_{n-1} \ra$ is good.  By bounded $\Sigma^0_1$ comprehension, let $B$ be the set of all $\la t, b_0, s_0, \dots, b_{n-1}, s_{n-1} \ra \leq m$ such that $\la b_0, s_0, \dots, b_{n-1}, s_{n-1} \ra$ is not good due to the failure of~(ii).  Then, for each $\la t, b_0, s_0, \dots, b_{n-1}, s_{n-1} \ra \in B$,
\begin{align*}
(\exists k < n)(\exists i > b_k)(\exists s)(\mu([S_{\la b_0, \dots, b_k \ra, s}] \cup [U_{i, s}]) > q_k).
\end{align*}
By $\bso$ there is a bound $\ell$ such that, for each $\la t, b_0, s_0, \dots, b_{n-1}, s_{n-1} \ra \in B$, there are a $k < n$, an $i$ with $b_k < i < \ell$, and an $s < \ell$ such that $\mu([S_{\la b_0, \dots, b_k \ra, s}] \cup [U_{i, s}]) > q_k$.  Therefore $T_{\la t, b_0, s_0, \dots, b_{n-1}, s_{n-1} \ra}^\ell = \emptyset$ for each $\la t, b_0, s_0, \dots, b_{n-1}, s_{n-1} \ra \in B$.  We have established that if $\la t, b_0, s_0, \dots, b_{n-1}, s_{n-1} \ra \leq m$ and $\la b_0, s_0, \dots, b_{n-1}, s_{n-1} \ra$ is not good, then $T_{\la t, b_0, s_0, \dots, b_{n-1}, s_{n-1} \ra}^\ell = \emptyset$.  Therefore $\bigcup_{i \leq m}T_i^\ell = \bigcup_{i \in A}T_i^\ell$.  Now, let $\tilde{n}$ be greatest such that some $\la \tilde{t}, \tilde{b}_0, \tilde{s}_0, \dots, \tilde{b}_{\tilde{n}-1}, \tilde{s}_{\tilde{n}-1} \ra$ is in $A$, and fix a witnessing $\la \tilde{b}_0, \dots, \tilde{b}_{\tilde{n}-1} \ra$.  By Claim~\ref{claim-UniqueSeq}, $\la \tilde{b}_0, \dots, \tilde{b}_{\tilde{n}-1}\ra$ is the unique sequence of length $\tilde{n}$ for which there are $\tilde{s}_0, \dots, \tilde{s}_{\tilde{n}-1}$ such that $\la \tilde{b}_0, \tilde{s}_0, \dots, \tilde{b}_{\tilde{n}-1}, \tilde{s}_{\tilde{n}-1} \ra$ is good.  Therefore, for any $\la t, b_0, s_0, \dots, b_{n-1}, s_{n-1} \ra \in A$, it must be that $n \leq \tilde{n}$ (by the maximality of $\tilde{n}$) and $(\forall j < n)(b_j = \tilde{b}_j)$.  We thus have that if $\la t, b_0, s_0, \dots, b_{n-1}, s_{n-1} \ra \in A$, then
\begin{align*}
[T_{\la t, b_0, s_0, \dots, b_{n-1}, s_{n-1} \ra}] = [S_{\la b_0, \dots, b_{n-1} \ra, t}] \subseteq \MP{S}_{\la b_0, \dots, b_{n-1} \ra} \subseteq \MP{S}_{\la \tilde{b}_0, \dots, \tilde{b}_{\tilde{n}-1} \ra}.
\end{align*}
However, $\mu(\MP{S}_{\la \tilde{b}_0, \dots, \tilde{b}_{\tilde{n}-1} \ra}) \leq q_{\tilde{n}-1}$.  So if we increase $\ell$ so as to be greater than the length of every string in every $S_{\la b_0, \dots, b_{n-1} \ra, t}$ for every $\la t, b_0, s_0, \dots, b_{n-1}, s_{n-1} \ra \in A$, we have that
\begin{align*}
2^{-\ell}\left|\bigcup_{i \leq m}T_i^\ell\right| = 2^{-\ell}\left|\bigcup_{i \in A}T_i^\ell\right| \leq \mu(\MP{S}_{\la \tilde{b}_0, \dots, \tilde{b}_{\tilde{n}-1} \ra}) \leq q_{\tilde{n}-1} < p
\end{align*}
as desired.

To see that $\bigcap_{i \geq N}\MP{U}_i \subseteq \MP V$ for each $N \in \Nb$, fix $N$ and suppose that $X \in \bigcap_{i \geq N}\MP{U}_i$.  Let $\la b_0, s_0, \dots, b_N, s_N \ra$ be good (which exists because by Claim~\ref{claim-UniqueSeq} there are good sequences of arbitrary length).  Then
\begin{align*}
X \in \bigcap_{i \geq N}\MP{U}_i \subseteq \MP{U}_{N \ldots b_N} \subseteq \MP{S}_{\la b_0, \dots, b_N \ra}.
\end{align*}
Let $t$ be such that $X \in [S_{\la b_0, \dots, b_N \ra, t}]$.  Then $X \in [T_{\la t, b_0, \dots, b_N \ra}] \subseteq \MP V$ as desired.

Finally, we observe that the sequence of trees $(T_i)_{i \in \Nb}$, and therefore the set $\MP V$, is produced with the required uniformity.
\end{proof}

\begin{Theorem}
\begin{align*}
\rca \vdash \forall X \forall Z(\text{$X$ is $2$-random relative to $Z$} \imp \text{$X$ is infinitely often $C^Z$-incompressible}).
\end{align*}
Hence $\rca \vdash \mlrr{2} \imp \cinc$.
\end{Theorem}

\begin{proof}
We work in $\rca$ and show that for every $X$ and $Z$, if $X$ is eventually $C^Z$-compressible, then $X$ is not $2$-random relative to $Z$.

Suppose $X$ and $Z$ are sets where $X$ is eventually $C^Z$-compressible.  That is,
\begin{align*}
\forall b \forall^\infty i(C^Z(X \rst i) < i-b).
\end{align*}
We show that there is a $\Sigma^{0, Z}_2$-test capturing $X$ and therefore that $X$ is not $2$-random relative to $Z$.

Define a double-sequence of open sets $(\MP{U}_{b,i})_{b, i \in \Nb} \leqT Z$ by defining $U_{b,i,s}$ so that, for each $b$ and $i$, $\bigcup_{s \in \Nb}U_{b,i,s}$ is an enumeration of all $\sigma \in 2^i$ such that $C^Z(\sigma) < i - b$.  Then $\forall b \forall i(\mu(\MP{U}_{b,i}) \leq 2^{-b})$ because there are at most $2^{i-b}$ strings $\sigma$ with $C^Z(\sigma) < i-b$.  Thus, for each fixed $b \in \Nb$, $(\MP{U}_{b,i})_{i \in \Nb}$ is a sequence of open sets such that $\forall i(\mu(\MP{U}_{b,i}) \leq 2^{-b})$.  Therefore, by the uniformity in Lemma~\ref{lem-2MLRimpINCHelper}, there is a sequence $(\MP{V}_b)_{b \in \Nb} \leqT Z$ of $\Sigma^{0, Z}_2$ sets such that $\forall b(\mu(\MP{V}_b) \leq 2^{-b+1})$ and $\forall N(\bigcap_{i \geq N}\MP{U}_{b, i} \subseteq \MP{V}_b)$.  The sequence $(\MP{V}_{b+1})_{b \in \Nb}$ is thus a $\Sigma^{0, Z}_2$ test.  We show that it captures $X$.  Given $b$, let $N$ be such that $(\forall i \geq N)[C^Z(X \rst i) < i-(b+1)]$.  Then $(\forall i \geq N)(X \in \MP{U}_{b+1, i})$.  Thus $X \in \bigcap_{i \geq N}\MP{U}_{b+1, i} \subseteq \MP{V}_{b+1}$ as desired.
\end{proof}

\begin{Corollary}
\begin{align*}
\rca \vdash \forall X \forall Z(\text{$X$ is infinitely often $C^Z$-incompressible} \biimp \text{$X$ is $2$-random relative to $Z$}).
\end{align*}
Hence $\cinc$ and $\mlrr{2}$ are equivalent over $\rca$.
\end{Corollary}

\section{Implications between major randomness notions within $\rca$}\label{sec-Implications}
Recall  the implications of randomness notions
\begin{center}   $2$-random $\Rightarrow$ weakly $2$-random  $\Rightarrow$  $1$-random  $\Rightarrow$ computably random $\Rightarrow$ Schnorr random.  \end{center}
  In this section, we show that the  implications between the corresponding principles  are   provable in $\rca$. We first provide the definitions of Schnorr and computable randomness.   For a   Schnorr test one requires   that the $n$\textsuperscript{th} component of the test has  measure exactly $2^{-n}$.

\begin{Definition}[$\rca$]
A \emph{Schnorr test relative to $Z$} is a Martin-L\"of test $(\MP{U}_n)_{n \in \Nb}$ relative to $Z$ where additionally the measures of the components of the test are uniformly computable from $Z$: $(\mu(\MP{U}_n))_{n \in \Nb} \leqT Z$.  $X$ is \emph{Schnorr random relative to $Z$} if $X \notin \bigcap_{n \in \Nb}\MP{U}_n$ for every Schnorr test $(\MP{U}_n)_{n \in \Nb}$ relative to $Z$.  $\sran$ is the statement ``for every $Z$ there is an $X$ that is Schnorr random relative to $Z$.''
\end{Definition}

For the purpose of defining Schnorr randomness relative to a set $Z$, we may assume that if $(\MP{U}_n)_{n \in \Nb}$ is a Schnorr test relative to $Z$, then $\mu(\MP{U}_n) = 2^{-n}$ for every $n$.  It is straightforward to implement the usual proof of this fact (see~\cite[Proposition~7.1.6]{DHBook}, for example) in $\rca$.

Computable randomness is defined in terms of computable betting strategies. They are called supermartingales in this context.

\begin{Definition}[$\rca$]{\ }
  A function $S \colon 2^{<\Nb} \imp \Qb^{\geq 0}$ is called a \emph{supermartingale} if 
\begin{align*}
(\forall \sigma \in 2^{<\Nb})(S(\sigma^\smf 0) + S(\sigma^\smf 1) \leq 2S(\sigma)),
\end{align*}
and it is called a \emph{martingale} if the defining property always holds with equality.
 A supermartingale $S$ \emph{succeeds} on a set $X$ if $\forall k \exists n(S(X \rst n) > k)$.
  $X$ is \emph{computably random relative to $Z$} if there is no supermartingale $S \leqT Z$ that succeeds on $X$.  $\cran$ is the statement ``for every $Z$ there is an $X$ that is computably random relative to $Z$.''
 
\end{Definition}

By \cite[Propositions~7.1.6 and~7.3.8]{NiesBook}, it makes no difference whether computable randomness relative to $Z$ is defined in terms of
\begin{itemize}
\item supermartingales $S \colon 2^{<\Nb} \imp \Qb^{\geq 0}$ that are $\leqT Z$;
\item supermartingales $S\colon 2^{<\Nb} \imp \Rb^{\geq 0}$ that are $\leqT Z$;
\item martingales $M \colon 2^{<\Nb} \imp \Qb^{\geq 0}$ that are $\leqT Z$; or
\item martingales $M \colon 2^{<\Nb} \imp \Rb^{\geq 0}$ that are $\leqT Z$.
\end{itemize}
It is straightforward to formalize these arguments in $\rca$.  In this setting, a function $S \colon 2^{<\Nb} \imp \Rb^{\geq 0}$ is coded by the corresponding sequence of values $(S(\sigma))_{\sigma \in 2^{<\Nb}}$.

\begin{Proposition}\label{prop-BasicRandImp}{\ }
\begin{itemize}
\item[(i)] $\rca \vdash \forall X \forall Z(\text{$X$ is $2$-random relative to $Z$} \imp \text{$X$ is weakly $2$-random relative to $Z$})$.

  Hence $\rca \vdash \mlrr{2} \imp \wtr$.

\medskip

\item[(ii)] $\rca \vdash \forall X \forall Z(\text{$X$ is weakly $2$-random relative to $Z$} \imp \text{$X$ is $1$-random relative to $Z$})$.  

Hence $\rca \vdash \wtr \imp \mlr$.

\medskip

\item[(iii)] $\rca \vdash \forall X \forall Z(\text{$X$ is $1$-random relative to $Z$} \imp \text{$X$ is computably random relative to $Z$})$.  

Hence $\rca \vdash \mlr \imp \cran$.

\medskip

\item[(iv)] $\rca \vdash \forall X \forall Z(\text{$X$ is computably random rel.\ to $Z$} \imp \text{$X$ is Schnorr random rel.\ to $Z$})$.

Hence $\rca \vdash \cran \imp \sran$.
\end{itemize}
\end{Proposition}

\begin{proof}
(i) To prove  that every $2$-random set is weakly $2$-random, one views $2$-randomness as $1$-randomness relative to $\emptyset'$ and shows that every weak $2$-test can be thinned to a Martin-L\"of test relative to $\emptyset'$ because $\emptyset'$ can uniformly compute the measures of $\Pi^0_1$ classes.  However, our formulation of $2$-randomness in $\rca$ is in terms of $\Sigma^0_2$-tests, so we need a version of this argument that works with $\Sigma^0_2$-tests instead of with Martin-L\"of tests relative to $0'$.

Let $(\MP{U}_n)_{n \in \Nb}$ be a weak $2$-test relative to $Z$, and let $(U_{n,i})_{n,i \in \Nb} \leqT Z$ be a code for $(\MP{U}_n)_{n \in \Nb}$.  For notational ease, assume that $\forall n \forall i(U_{n,i} \subseteq U_{n,i+1})$.  We define a double-sequence $(T_{n,i})_{n,i \in \Nb} \leqT Z$ of trees coding a $\Sigma^{0,Z}_2$-test $(\MP{W}_n)_{n \in \Nb}$ such that $\bigcap_{n \in \Nb} \MP{U}_n \subseteq \bigcap_{n \in \Nb} \MP{W}_n$.  The idea is to take $\MP{W}_n = \bigcup_{i \in \Nb}[T_{n, i}]$ to be $\MP{U}_m$ for the least $m$ such that $\mu(\MP{U}_m) \leq 2^{-n}$.  To do this, we view each $i$ as a triple $i = \la \sigma, m, s \ra$ and use the trees $T_{n, \la \sigma, m, s \ra}$ to ensure that $[\sigma] \subseteq \MP{W}_n$ when $[\sigma] \subseteq \MP{U}_m$ and $m$ is least such that $\mu(\MP{U}_m) \leq 2^{-n}$.

 To define $T_{n, \la \sigma, m, s \ra}$, first check that $\sigma \in U_{m,s}$ and that $s$ is large enough to witness that $\mu(\MP{U}_k) > 2^{-n}$ for all $k < m$.  If one of the  checks fails,   set $T_{n, \la \sigma, m, s \ra} = \emptyset$.  If both checks pass, then $[\sigma] \subseteq \MP{U}_m$, and possibly $m$ is least such that $\mu(\MP{U}_m) \leq 2^{-n}$.  In this case, start adding to $T_{n, \la \sigma, m, s \ra}$ all strings comparable with $\sigma$, level-by-level, until possibly seeing that $\mu(\MP{U}_m) > 2^{-n}$.  In the end, if $[\sigma] \subseteq \MP{U}_m$, $m$ is least such that $\mu(\MP{U}_m) \leq 2^{-n}$, and $s$ is large enough, then $T_{n, \la \sigma, m, s \ra}$ consists of all strings comparable with $\sigma$, so $[T_{n, \la \sigma, m, s \ra}] = [\sigma]$.  Otherwise, $T_{n, \la \sigma, m, s \ra}$ is finite, so $[T_{n, \la \sigma, m, s \ra}] = \emptyset$.  Therefore $\MP{W}_n = \MP{U}_m$.

Formally, for each $n$ and $\la \sigma, m, s \ra$, define $T_{n, \la \sigma, m, s \ra}$ so that
\begin{align*}
[T_{n, \la \sigma, m, s\ra}] = 
\begin{cases}
[\sigma] & \text{if $\sigma \in U_{m,s}$, $\mu(\MP{U}_m) \leq 2^{-n}$, and $(\forall k < m)(\mu(U_{k,s}) > 2^{-n})$}\\
\emptyset & \text{otherwise}.
\end{cases}
\end{align*}

To do this, if $\sigma \notin U_{m,s}$ or $(\exists k < m)(\mu(U_{k,s}) \leq 2^{-n})$, then let $T_{n, \la \sigma, m, s \ra} = \emptyset$.  Otherwise, let $T_{n, \la \sigma, m, s \ra}$ be the set of all strings $\tau \in 2^{< \Nb}$ such that $\tau$ is comparable with $\sigma$ (i.e., either $\tau \subseteq \sigma$ or $\tau \supseteq \sigma$) and $\mu(U_{m, |\tau|}) \leq 2^{-n}$.  Observe that $(T_{n,i})_{n,i \in \Nb} \leqT Z$ because $(U_{n,i})_{n,i \in \Nb} \leqT Z$.  Let $(\MP{W}_n)_{n \in \Nb}$ denote the uniform sequence of $\Sigma^{0,Z}_2$ sets defined by $(T_{n,i})_{n,i \in \Nb}$.

Fix $n$.  We show that there is a least $m$ such that $\mu(\MP{U}_m) \leq 2^{-n}$, that $\MP{U}_m \subseteq \MP{W}_n$, and that $\mu(\MP{W}_n) \leq 2^{-n}$.

To see that there is a least $m$ such that $\mu(\MP{U}_m) \leq 2^{-n}$, first observe that there is some $m$ such that $\mu(\MP{U}_m) \leq 2^{-n}$ because $\lim_m \mu(\MP{U}_m) = 0$ by the fact that $(\MP{U}_n)_{n \in \Nb}$ is a weak $2$-test.  Thus there is a least such $m$ by the $\Pi^0_1$ least element principle.  Henceforth $m$ always denotes the least $m$ such that $\mu(\MP{U}_m) \leq 2^{-n}$.

To show that $\MP{U}_m \subseteq \MP{W}_n$, we first show that there is a $t$ large enough to witness that $\mu(\MP{U}_k) > 2^{-n}$ for all $k < m$.  Once we have $t$, we argue that if $\sigma \in U_{m,s}$ for some $s$, then $\sigma \in U_{m,s}$ for some $s > t$ (as we assume that the $U_{m,s}$'s are nested), in which case $[\sigma] = [T_{n, \la \sigma, m, s \ra}] \subseteq \MP{W}_n$.  Formally, because $m$ is least, we have that $(\forall k < m)(\mu(\MP{U}_k) > 2^{-n})$ and hence that $(\forall k < m)(\exists t)(\mu(U_{k,t}) > 2^{-n})$.  By $\bso$, there is a fixed $t$ such that $(\forall k < m)(\mu(U_{k,t}) > 2^{-n})$.  Now, suppose that $Y \in \MP{U}_m$, and let $\sigma \subseteq Y$ and $s > t$ be such that $\sigma \in U_{m,s}$.  Then $[T_{n, \la \sigma, m, s \ra}] = [\sigma]$, so $Y \in [T_{n, \la \sigma, m, s \ra}] \subseteq \MP{W}_n$.  Thus $\MP{U}_m \subseteq \MP{W}_n$.

To show that $\mu(\MP{W}_n) \leq 2^{-n}$, we need to show that $\forall i \exists b (2^{-b}|\bigcup_{j \leq i}T_{n,j}^b| \leq 2^{-n})$.  Fix $i$.  We find a $b$ large enough so that each string in $\bigcup_{j \leq i}T_{n,j}^b$ is an extension of some string in $\bigcup_{s \in \Nb} U_{m,s}$.  This is achieved by choosing $b$ to be greater than $|\sigma|$ for every $\la \sigma, k, s \ra \leq i$ and greater than the length of every string in the finite trees $T_{n, \la \sigma, k, s \ra}$ with $\la \sigma, k, s \ra \leq i$.  Once we have $b$,  since $\mu(\MP{U}_m) \leq 2^{-n}$ it follows that $2^{-b}|\bigcup_{j \leq i}T_{n,j}^b| \leq 2^{-n}$.

As above, let $t$ be such that $(\forall k < m)(\mu(U_{k,t}) > 2^{-n})$.  Let $b > \max\{t, i\}$ (so that if $\la \sigma, k, s \ra \leq i$, then $b > |\sigma|$).  We show that this $b$ is large enough.  Consider a $\la \sigma, k, s \ra \leq i$.  If $k < m$, then $T_{n, \la \sigma, k, s \ra}^t = \emptyset$ because $\mu(U_{k,t}) > 2^{-n}$ by choice of $t$.  If $k > m$, then $T_{n, \la \sigma, k, s \ra} = \emptyset$ because $\mu(U_{m,s}) \leq 2^{-n}$.  So if $\tau \in T_{n, \la \sigma, k, s \ra}^b$ for $\la \sigma, k, s \ra \leq i$, then it must be that $k = m$, $\sigma \in U_{m,s}$, and $\tau \supseteq \sigma$.  Thus every string in $\bigcup_{j \leq i}T_{n,j}^b$ is an extension of some string in $\bigcup_{s \in \Nb} U_{m,s}$.  Therefore $2^{-b}|\bigcup_{j \leq i}T_{n,j}^b| \leq \mu(\MP{U}_m) \leq 2^{-n}$.

Now, suppose that $X$ is not weakly $2$-random relative to $Z$.  Then there is a weak $2$-test $(\MP{U}_n)_{n \in \Nb}$ relative to $Z$ that captures $X$.  By the above, there is a $\Sigma^{0,Z}_2$-test $(\MP{W}_n)_{n \in \Nb}$ such that $X \in \bigcap_{n \in \Nb} \MP{U}_n \subseteq \bigcap_{n \in \Nb} \MP{W}_n$.  Therefore $X$ is not $2$-random relative to $Z$.

\medskip
(ii) This is immediate from the definitions because every Martin-L\"of test relative to $Z$ is also a weak $2$-test relative to $Z$.
\medskip

(iii) See the proof of the (i)$\Imp$(iii) implication of~\cite[Proposition~7.2.6]{NiesBook}, which is straightforward to formalize in $\rca$. 

\medskip

(iv) See the proof of~\cite[Proposition~7.3.2]{NiesBook}, which is straightforward to formalize in $\rca$.  Note however that this proof makes use of $\Rb^{\geq 0}$-valued martingales.
\end{proof}

Not every Schnorr random set is computably random (see for example~\cite[Theorem 7.5.10]{NiesBook}).  However, it is provable in $\rca$ that if a Schnorr random set exists, then a computably random set exists.  Thus computable randomness and Schnorr randomness are equivalent as set-existence axioms.

\begin{Theorem}\label{thm-SRimpCR}
$\rca \vdash \sran \imp \cran$.  Thus $\sran$ and $\cran$ are equivalent over $\rca$.
\end{Theorem}

\begin{proof}
Assume $\sran$.  Let $Z$ be given.  We want to find a set $X$ that is computably random relative to $Z$.  By $\sran$, let $Y$ be Schnorr random relative to $Z$.  If $Y$ is $1$-random relative to $Z$, then it is also computably random relative to $Z$ by Proposition~\ref{prop-BasicRandImp} and we are done.  Otherwise, $Y$ is not $1$-random relative to $Z$, so there is a $\Sigma^{0,Z}_1$-test $(\MP{U}_n)_{n \in \Nb}$ with $Y \in \bigcap_{n \in \Nb}\MP{U}_n$.  Let $(B_{n,i})_{n,i \in \Nb}$ denote the code for $(\MP{U}_n)_{n \in \Nb}$.  For notational ease, assume that $\forall n \forall i(B_{n,i} \subseteq B_{n,i+1})$.  Define $f \colon \Nb \imp \Nb$ by
\begin{align*}
f(n) = \text{the least $i$ such that $(\exists \sigma \in B_{n,i})(\sigma \subseteq Y)$}
\end{align*}
(recall that each $B_{n,i}$ is finite, so $f$ can be defined in $\rca$).

For functions $f, g \colon \Nb \imp \Nb$, say that \emph{$f$ eventually dominates $g$} if $(\exists n)(\forall k > n)(g(k) < f(k))$.

\begin{Claim}\label{claim-ZDom}
If $g \colon \Nb \imp \Nb$ is a function with $g \leqT Z$, then $f$ eventually dominates $g$.
\end{Claim}

\begin{proof}[Proof of Claim]
Suppose for a contradiction that there is a $g \leqT Z$ that is not eventually dominated by $f$.  Define a uniform sequence of $\Sigma^{0,Z}_1$ sets $(\MP{V}_n)_{n \in \Nb}$ coded by $(C_{n,m})_{n,m \in \Nb} \leqT Z$ by letting

\begin{align*}
C_{n,m} = 
\begin{cases}
\emptyset & \text{if $n \geq m$}\\
\text{a finite $C \supseteq C_{n,m-1} \cup B_{m,g(m)}$ with $\mu(C) = 2^{-n}-2^{-m}$} & \text{if $n < m$}.
\end{cases}
\end{align*}
This is possible because if $n < m$ and $\mu(C_{n,m-1}) = 2^{-n} - 2^{-(m-1)}$, then
\begin{align*}
\mu(C_{n, m-1} \cup B_{m,g(m)}) \leq 2^{-n} - 2^{-(m-1)} + 2^{-m} = 2^{-n} - 2^{-m},
\end{align*}
so such a $C_{n,m}$ exists.  We have that $\forall n (\mu(\MP{V}_n) = 2^{-n})$, so $(\MP{V}_n)_{n \in \Nb}$ is a Schnorr test relative to $Z$.  Furthermore, this test captures $Y$ because if $m > n$ and $f(m) \leq g(m)$, then $Y \in [B_{m,g(m)}] \subseteq [C_{n,m}] \subseteq \MP{V}_n$.  This contradicts that $Y$ is Schnorr random relative to $Z$.
\end{proof}

The rest of the proof follows the usual proof that every high set computes a computably random set (see e.g.\ \cite[Theorem~7.5.2]{NiesBook}).  We use $f$ to define a supermartingale that multiplicatively dominates every supermartingale $\leqT Z$.  In the following, all supermartingales are $\Qb^{\geq 0}$-valued.

First, using our effective sequence $(\Phi_e)_{e \in \Nb}$ of all Turing functionals, define a sequence of Turing functionals $(\Psi_e)_{e \in \Nb}$ such that
  $\Psi_e^Z$ always computes a partial supermartingale, and
  if $\Phi_e^Z$ is total and computes a supermartingale, then $\forall \sigma(\Phi_e^Z(\sigma) = \Psi_e^Z(\sigma))$.
This can be accomplished by setting
\begin{align*}
\Psi_e^Z(\emptyset) &= \Phi_e^Z(\emptyset)\\
\Psi_e^Z(\sigma^\smf a) &=
\begin{cases}
\Phi_e^Z(\sigma^\smf a) & \text{if $\Phi_e^Z(\sigma)\da$, $\Phi_e^Z(\sigma^\smf0)\da$, $\Phi_e^Z(\sigma^\smf1)\da$, and $\Phi_e^Z(\sigma^\smf0) + \Phi_e^Z(\sigma^\smf1) \leq 2\Phi_e^Z(\sigma)$}\\
\ua & \text{otherwise},
\end{cases}
\end{align*}
for $a \in \{0,1\}$.
Now define a sequence of Turing functionals $(\Gamma_e)_{e \in \Nb}$ such that
  $\Gamma_e^Z$ always computes a partial supermartingale,
  $\Gamma_e^Z(|\sigma|) = 1$ if $|\sigma| \leq e$, and
  if $\Phi_e^Z$ is total and computes a supermartingale, then there is a $c \in \Nb$ such that $\forall \sigma(\Phi_e^Z(\sigma) \leq c\Gamma_e^Z(\sigma))$.
This can be accomplished by setting
\begin{align*}
\Gamma_e^Z(\sigma) = 
\begin{cases}
1 & \text{if $|\sigma| \leq e$}\\
0 & \text{if $|\sigma| > e$ and $\Psi_e^Z(\sigma \rst e)\da = 0$}\\
\Psi_e^Z(\sigma)/\Psi_e^Z(\sigma \rst e) & \text{if $|\sigma| > e$, $\Psi_e^Z(\sigma)\da$, and $\Psi_e^Z(\sigma \rst e)\da > 0$}\\
\ua & \text{otherwise}.
\end{cases}
\end{align*}
If $\Phi_e^Z$ is total and computes a supermartingale, let $c > \max\{\Phi_e^Z(\sigma) : \sigma \in 2^e\}$.  Then $\forall \sigma(\Phi_e^Z(\sigma) \leq c\Gamma_e^Z(\sigma))$.

Assemble a supermartingale from $Z$ and $f$ as follows.  First, for each $e \in \Nb$, let
\begin{align*}
S_e(\sigma) =
\begin{cases}
\Gamma_e^Z(\sigma) & \text{if $|\sigma| \leq e$ or $(\forall \tau \in 2^{\leq|\sigma|})(\text{$\Gamma_e^Z(\tau)$ halts within $f(|\sigma|) + e$ steps})$}\\
0 & \text{otherwise.}
\end{cases}
\end{align*}
Now let $S(\sigma) = \sum_{e \in \Nb}2^{-e}S_e(\sigma)$.  Notice that $S$ is $\Qb^{\geq 0}$-valued because $S_e(\sigma) = 1$ when $e \geq |\sigma|$, so $\sum_{e \geq |\sigma|}2^{-e}S_e(\sigma) = 2^{-e+1}$.  One may then verify that each $S_e$ is a supermartingale and therefore that $S$ is a supermartingale.

Suppose that $P \leqT Z$ is a supermartingale.  We show that there is a $d \in \Nb$ such that $\forall  \sigma(P(\sigma) \leq dS(\sigma))$.  Let $e_0$ be such that $P = \Phi_{e_0}^Z$.  Then $\Gamma_{e_0}^Z$ is total, so define the total function $g \leqT Z$ by
\begin{align*}
g(n) = \text{the least $t$ such that $(\forall \tau \in 2^{\leq n})(\text{$\Gamma_{e_0}^Z(\tau)$ halts within $t$ steps})$}.
\end{align*}
By Claim~\ref{claim-ZDom}, there is an $n \in \Nb$ such that $(\forall k > n)(g(k) < f(k))$.  By padding, let $e_1 > \max\{g(m) : m \leq n\}$ be such that $\Gamma_{e_1}$ is the same functional as $\Gamma_{e_0}$.  Then
\begin{align*}
(\forall k)(\forall \tau \in 2^{\leq k})(\text{$\Gamma_{e_1}^Z(\tau)$ halts within $f(k) + e_1$ steps}),
\end{align*}
and therefore $\forall \sigma (S_{e_1}(\sigma) = \Gamma_{e_1}^Z(\sigma))$.  Let $c$ be such that $\forall \sigma(P(\sigma) \leq c\Gamma_{e_1}^Z(\sigma))$.  Let $d = c2^{e_1}$.  Then, for all $\sigma \in 2^{<\Nb}$, 
\begin{align*}
P(\sigma) \leq c\Gamma_{e_1}^Z(\sigma) = cS_{e_1}(\sigma) \leq c2^{e_1}S(\sigma) = dS(\sigma),
\end{align*}
as desired.

To finish the proof, let $X$ be the leftmost non-ascending path of $S$.  That is, define $X = \lim_{s}\sigma_s$ recursively by $\sigma_0 = \emptyset$ and
\begin{align*}
\sigma_{s+1} = 
\begin{cases}
{\sigma_s}^\smf 0 & \text{if $S({\sigma_s}^\smf 0) \leq S(\sigma_s)$}\\
{\sigma_s}^\smf 1 & \text{otherwise}.
\end{cases}
\end{align*}
If $P \leqT Z$ is a supermartingale, there is a $d \in \Nb$ such that $\forall \sigma(P(\sigma) \leq dS(\sigma))$.  So for all $n \in \Nb$, $P(X \rst n) \leq dS(X \rst n) \leq dS(\emptyset)$.  Thus $P$ does not succeed on $X$.  Thus no supermartingale $P \leqT Z$ succeeds on $X$, so $X$ is computably random relative to $Z$.
\end{proof}

\section{Weak Demuth and balanced randomness} \label{s:wD}

Randomness notions that have been introduced only recently include  $h$-weak Demuth randomness for an order function $h$ as well as the special case of balanced randomness, where $h(n) $ can be taken to be $2^n$  \cite[Section 7]{FigueiraHirschfeldtMillerNgNies}.   An $h$-Demuth test is like a Martin-L\"of test, except that we allow ourselves to change the index  of the $n$\textsuperscript{th} component of the test up to $h(n)$ many times.  To make this precise, we must first define codes for $h$-r.e.\ functions.

\begin{Definition}[$\rca$] \label{df:balanced}
Let $h \colon \Nb \imp \Nb$.  A (coded) \emph{$h$-r.e.\ function} is a function $g \colon \Nb \times \Nb \imp \Nb$ such that $|\{s : g(n,s) \neq g(n,s+1)\}| \leq h(n)$ for every $n \in \Nb$.  If also $h, g \leqT Z$ for some set $Z$, then we say that $g$ is a (coded) \emph{$h$-r.e.\ function relative to $Z$}.
\end{Definition}

If $g$ codes an $h$-r.e.\ function, then $\rca$ proves that the limit $\lim_s g(n,s)$ exists for each individual $n$, but it does not prove that there is always a function $f$ such that $\forall n (f(n) = \lim_s g(n,s))$.

 \begin{Definition}[$\rca$]{\ }
\begin{itemize}
\item Let $h \leqT Z$.  A code for an \emph{$h$-Demuth test relative to $Z$} is a coded $h$-r.e.\ function $g \leqT Z$ where, for all $n \in \Nb$, $e_n = \lim_s g(n,s)$ is an index such that $\Phi_{e_n}^Z$ computes a code for a $\Sigma^{0,Z}_1$ set $\MP{U}_n$ with $\mu(\MP{U}_n) \leq 2^{-n}$.

\medskip

\item A set $X$ \emph{weakly passes} the $h$-Demuth test relative to $Z$ coded by $g$ if there is an $n \in \Nb$ such that $X \notin \MP{U}_n$, where, as above, $\MP{U}_n$ is the $\Sigma^{0,Z}_1$ set coded by $\Phi_{e_n}^Z$ for $e_n = \lim_s g(n,s)$.

\medskip

\item For $h \leqT Z$, $X$ is \emph{$h$-weakly Demuth random relative to $Z$} if $X$ weakly passes every $h$-Demuth test relative to $Z$. These definitions are sometimes extended to classes of order functions in the expected way.

\medskip

\item $X$ is \emph{balanced random relative to $Z$} if $X$ weakly passes every $O(2^n)$-Demuth test relative to $Z$ (that is, if, for every $c \in \Nb$, $X$ weakly passes every $c2^n$-Demuth test relative to $Z$). 

\medskip

\item Let $h$ be a function that is provably total in $\rca$.  Then $\wdrr{h}$ is the statement ``for every $Z$ there is an $X$ that is $h$-weakly Demuth random relative to $Z$.''

\medskip

\item $\bran$ is the statement ``for every $Z$ there is an $X$ that is balanced random relative to $Z$.''
\end{itemize}
\end{Definition}

Not every $1$-random set is balanced random.  For example, there are left-r.e.\ $1$-random sets, but no left-r.e.\ set is balanced random.  However, if $X = X_0 \oplus X_1$ is $1$-random, then either $X_0$ is balanced random or $X_1$ is balanced random.  This fact follows from the more elaborate \cite[Theorem 23]{FigueiraHirschfeldtMillerNgNies}, which states that a $1$-random set $X$ is $O(h(n)2^n)$-weakly Demuth random for some order function $h$ if and only $X$ it is not $\omega$-r.e.-tracing (roughly, $X$ is $\omega$-r.e.-tracing if  for each $\omega$-r.e.\ function there is an $X$-r.e.\ trace of a fixed size bound).  We sketch the argument.  Suppose that $X = X_0 \oplus X_1$ is $1$-random.  Then $X_0$ is $1$-random and $X_1$ is $1$-random relative to $X_0$ by van Lambalgen's theorem.  If $X_0$ is not $\omega$-r.e.-tracing, then, by \cite[Theorem 23]{FigueiraHirschfeldtMillerNgNies}, $X_0$ is $O(h(n)2^n)$-weakly Demuth random for some order function $h$, and therefore $X_0$ is balanced random.  On the other hand, if $X_0$ is $\omega$-r.e.-tracing, then every $O(2^n)$-Demuth test can be converted into a $\Sigma^{0, X_0}_1$-test.  So if $X_1$ were not balanced random, then $X_1$ would not be $1$-random relative to $X_0$, which contradicts van Lambalgen's theorem.  Thus, in this case, $X_1$ must be balanced random.

We now give a direct proof that if $X = X_0 \oplus X_1$ is $1$-random, then either $X_0$ or $X_1$ is balanced random.  This proof avoids considering $\omega$-r.e.-traceability and is easy to formalize in $\rca$.

\begin{Theorem}\label{thm-MLRimpBran}
$\rca \vdash \mlr \imp \bran$.  Thus $\mlr$ and $\bran$ are equivalent over $\rca$.
\end{Theorem}

\begin{proof}
Assume $\mlr$.  Let $Z$ be given.  We want to find a set that is balanced random relative to $Z$.  Let $X = X_0 \oplus X_1$ be $1$-random relative to $Z$.  We show that one of  $X_0$, $X_1$ is balanced random relative to $Z$.  Assume otherwise.  Let $g_0, g_1 \colon \Nb \times \Nb \imp \Nb$ be codes for $c2^n$-Demuth tests (for some $c \in \Nb$) relative to $Z$ capturing $X_0$ and $X_1$, respectively.  By modifying $g_0$ and $g_1$ if necessary, we may assume that $\Phi_{g_0(n,s)}^Z$ and $\Phi_{g_1(n,s)}^Z$ both compute codes of $\Sigma^{0,Z}_1$ sets $\MP{U}^0_{n,s}$ and $\MP{U}^1_{n,s}$ of measure $\leq 2^{-n}$ for all $n$ and $s$.  We may also assume that $g_0(n, \cdot)$ and $g_1(n, \cdot)$ change at least once for each $n$ (by increasing $c$ by $1$ and adding dummy changes, if necessary).

We define a $\Sigma^{0,Z}_1$-test capturing $X$, contradicting that $X$ is $1$-random relative to $Z$.  If $g_0$ does not change last on infinitely many $n$, then $g_1$ changes last on infinitely many $n$.  So suppose for the sake of argument that $g_0$ changes last on infinitely many $n$.  Define a uniform sequence $(\MP{O}_n)_{n \in \Nb}$ of $\Sigma^{0,Z}_1$ sets by letting
\begin{align*}
\MP{O}_n = \bigcup_{\substack{s > 0\\g_0(n,s) \neq g_0(n,s-1)}} \MP{U}^0_{n,s} \oplus \MP{U}^1_{n,s}
\end{align*}
for each $n$.  Here, for $\mbf{\Sigma^0_1}$ sets $\MP{A}_0$ and $\MP{A}_1$, $\MP{A}_0 \oplus \MP{A}_1$ denotes the $\mbf{\Sigma^0_1}$ set of all $Y = Y_0 \oplus Y_1$ with $Y_0 \in \MP{A}_0$ and $Y_1 \in \MP{A}_1$.  For $\mbf{\Sigma^0_1}$ sets $\MP{A}_0$ and $\MP{A}_1$, it is straightforward to produce a code for $\MP{A}_0 \oplus \MP{A}_1$ and to show that if $\mu(\MP{A}_0) \leq a_0$ and $\mu(\MP{A}_1) \leq a_1$, then $\mu(\MP{A}_0 \oplus \MP{A}_1) \leq a_0a_1$.  So $\mu(\MP{U}^0_{n,s} \oplus \MP{U}^1_{n,s}) \leq 2^{-2n}$ for all $n$ and $s$.  Each $\MP{O}_n$ is the union of at most $c2^n$ sets (because $g_0$ is $c2^n$-r.e.) of measure at most $2^{-2n}$ each.  Therefore $\mu(\MP{O}_n) \leq c2^{-n}$ for each $n$.  Now, choose $k$ such that $2^k > c$.  Define another uniform sequence $(\MP{V}_n)_{n \in \Nb}$ of $\Sigma^{0,Z}_1$ sets by letting $\MP{V}_n = \bigcup_{i > n+k}\MP{O}_i$ for each $n$.  Then $\mu(\MP{V}_n) \leq c2^{-n-k} \leq 2^{-n}$ for each $n$, so $(\MP{V}_n)_{n \in \Nb}$ is a $\Sigma^{0,Z}_1$-test.

We claim that $X \in \bigcap_{n \in \Nb}\MP{V}_n$.  It suffices to show that, for every $n$, there is an $i > n+k$ with $X \in \MP{O}_i$.  By the assumption on $g_0$, let $i > n+k$ be such that there is an $s_0 > 0$ such that $g_0(i, s_0) \neq g_0(i, s_0-1)$ and $(\forall t > s_0)(g_1(i, t) = g_1(i, s_0))$.  Let $s_0 > 0$ be greatest such that $g_0(i, s_0) \neq g_0(i, s_0-1)$.  Then $g_0(i,s_0) = \lim_s g_0(i,s)$ and $g_1(i,s_0) = \lim_s g_1(i,s)$, so $X_0 \in \MP{U}^0_{i,s_0}$ and $X_1 \in \MP{U}^1_{i,s_0}$ because the $c2^n$-Demuth tests coded by $g_0$ and $g_1$ capture $X_0$ and $X_1$.  Thus
\begin{align*}
X = X_0 \oplus X_1 \in \MP{U}^0_{i,s_0} \oplus \MP{U}^1_{i,s_0} \subseteq \MP{O}_i
\end{align*}
as desired.
\end{proof}

\section{Non-implications via $\omega$-models}\label{sec-NonImp}

In this section, we exhibit $\omega$-models of $\rca$ that witness various non-implications between pairs of randomness-existence principles.  We also compare randomness-existence principles to principles asserting the existence of diagonally non-recursive functions.

\begin{Definition}[$\rca$]{\ }
 A function $f \colon \Nb \imp \Nb$ is \emph{diagonally non-recursive relative to $Z$} if $\forall e(\Phi_e^Z(e)\da \imp f(e) \neq \Phi_e^Z(e))$.
 $\dnr$ is the statement ``for every $Z$ there is an $f$ that is diagonally non-recursive relative to $Z$.''
\end{Definition}

$\dnr$ is a common benchmark by which to gauge the computability-theoretic strength of set-existence principles.  It is a well-known observation of Ku\v{c}era (see e.g.~\cite[Proposition~4.1.2]{NiesBook}) that every $1$-random set computes a diagonally non-recursive function.  By formalizing this result, one readily sees that $\rca \vdash \mlr \imp \dnr$.  In contrast, $\cran$ is not strong enough to produce diagonally non-recursive functions.

\begin{Proposition}
There is an $\omega$-model of $\rca + \cran + \neg\dnr$.  Thus $\rca \nvdash \cran \imp \dnr$, and therefore also $\rca \nvdash \cran \imp \mlr$.
\end{Proposition}

\begin{proof}
For the purposes of this proof, say that a set $A$ is \emph{high} relative to a set $B$ if there is a single function $f \leqT A$ that eventually dominates every function $g \leqT B$.  We apply  the following two results.
\begin{itemize}
\item[(i)] If $A$ is high relative to $B$, then $A \oplus B$ computes a set that is computably random relative to $B$ (see \cite[Theorem~7.5.2]{NiesBook}; the proof is also replicated in the proof of Theorem~\ref{thm-SRimpCR}).

\medskip

\item[(ii)] If $B$ does not compute a diagonally non-recursive function, then there is an $A$ that is high relative to $B$ such that $A \oplus B$ does not compute a diagonally non-recursive function~\cite[Lemma~4.14]{Cholak:2006eg}.
\end{itemize}
\noindent  By iterating result~(ii) in the usual way, we produce an $\omega$-model $\mf M = (\omega, \MP S)$ of $\rca + \neg\dnr$ such that for every $B \in \MP S$ there is an $A \in \MP S$ that is high relative to $B$.  By result~(i), $\mf M \models \cran$.  Thus $\mf M \models \rca + \cran + \neg\dnr$.
\end{proof}

In order to give useful formalizations of stronger versions of $\dnr$, we must carefully express computations relative to $Z'$ for a set $Z$ without implying the existence of $Z'$ as a set.  We make the following definitions in $\rca$ (see~\cites{AvigadDeanRute,BienvenuPateyShafer}).

\begin{itemize}
\item Let $e \in Z'$ abbreviate the formula $\Phi_e^Z(e)\da$.

\medskip

\item Let $\sigma \subseteq Z'$ abbreviate the formula $(\forall e < |\sigma|)(\sigma(e) = 1 \biimp e \in Z')$.

\medskip

\item Let $\Phi_e^{Z'}(x) = y$ abbreviate the formula $(\exists \sigma \subseteq Z')(\Phi_e^\sigma(x) = y)$.  Similarly, let $\Phi_e^{Z'}(x)\da$ denote that there is a $y$ such that $\Phi_e^{Z'}(x) = y$.
\end{itemize}
Notice that, by bounded $\Sigma^0_1$ comprehension, $\rca$ proves that the set $\{e < n : e \in Z'\}$ exists for every $Z$ and $n$.  Then by letting $\sigma$ be the characteristic string of $\{e < n : e \in Z'\}$, we see that $\rca$ proves that for every $Z$ and $n$ there is a $\sigma$ of length $n$ such that $(\forall e < |\sigma|)(\sigma(e) = 1 \biimp e \in Z')$.

\begin{Definition}[$\rca$]{\ }
 A function $f \colon \Nb \imp \Nb$ is \emph{diagonally non-recursive relative to $Z'$} for a set $Z$ if $\forall e(\Phi_e^{Z'}(e)\da \imp f(e) \neq \Phi_e^{Z'}(e))$.
  $\dnrr{2}$ is the statement ``for every $Z$ there is an $f$ that is diagonally non-recursive relative to $Z'$.''
\end{Definition}

Ku\v{c}era in fact showed that every $n$-random set computes a function that is diagonally non-recursive relative to $0^{(n-1)}$.  This result can be formalized in $\rca$ (see \cite[Theorem~2.8]{BienvenuPateyShafer}).  In particular, $\rca \vdash \mlrr{2} \imp \dnrr{2}$.  In contrast, $\wtr$ is not strong enough to produce diagonally non-recursive functions relative to $0'$.

\begin{Theorem}\label{thm-WTRnot2DNR}
There is an $\omega$-model of $\rca + \wtr + \neg\dnrr{2}$.  Thus $\rca \nvdash \wtr \imp \dnrr{2}$, and therefore also $\rca \nvdash \wtr \imp \mlrr{2}$.
\end{Theorem}

\begin{proof}
The intuition is to build a model of $\rca + \wtr + \neg\dnrr{2}$ out of the columns of a weakly $2$-random set $Z$ that does not compute a $\dnrr{2}$ function.  For this idea to work, $Z$ must be chosen with a little care because the relevant direction of van Lambalgen's theorem does not hold for weak $2$-randomness in general~\cite{BarmpaliasDowneyNg}.

Recall that a set $A$ has \emph{hyperimmune-free degree} (or is \emph{computably dominated}) if every $f \leqT A$ is eventually dominated by a computable function.  Let $Z$ be a $1$-random set of hyperimmune-free degree that does not compute a diagonally non-recursive function relative to $0'$.  Such a $Z$ exists by  \cite[Theorem 5.1]{Kucera.Nies:11} (also see~\cite[Exercise 1.8.46]{NiesBook}), which states that if $\MP C \subseteq 2^\omega$ is a non-empty $\Pi^0_1$ class and $B \geT 0'$ is $\Sigma^0_2$, then there is a $Z \in \MP C$ of hyperimmune-free degree with $Z' \leqT B$.  Let $\MP C \subseteq 2^\omega$ be a non-empty $\Pi^0_1$ class consisting entirely of $1$-randoms, and let $B$ be any set r.e.\ in $0'$ such that $0' \leT B \leT 0''$.  Let $Z \in \MP C$ be of hyperimmune-free degree such that $Z' \leqT B$.  Then of course $Z$ is $1$-random and has hyperimmune-free degree.  Furthermore, $Z$ does not compute a diagonally non-recursive function relative to $0'$.  If $Z$ computes a diagonally non-recursive function relative to $0'$, then so does $B$, but then we would have $B \geqT 0''$ by the Arslanov completeness criterion relative to $0'$, which is a contradiction.

Decompose $Z$ into columns $Z = \bigoplus_{n \in \omega}Z_n$, where $Z_n = \{k : \la n, k \ra \in Z\}$ for each $n$.  By a straightforward relativization of~\cite[Proposition~3.6.4]{NiesBook}, if $X \oplus Y$ has hyperimmune-free degree and $Y$ is $1$-random relative to $X$, then $Y$ is also weakly $2$-random relative to $X$.  It follows that $Z_{n+1}$ is weakly $2$-random relative to $\bigoplus_{i \leq n}Z_i$ for every $n$.  This is because $\bigoplus_{i \leq n+1}Z_i$ has hyperimmune-free degree (as $Z$ has hyperimmune-free degree) and $Z_{n+1}$ is $1$-random relative to $\bigoplus_{i \leq n}Z_i$ by van Lambalgen's theorem.

Let $\MP S = \{X : \exists n(X \leqT \bigoplus_{i \leq n}Z_i)\}$, and let $\mf M = (\omega, \MP S)$.  $\MP S$ contains no diagonally non-recursive function relative to $0'$, so $\mf M \models \rca + \neg\dnrr{2}$.  If $X \in \MP S$ and $n$ is such that $X \leqT \bigoplus_{i \leq n}Z_i$, then $Z_{n+1} \in \MP S$ is weakly $2$-random relative to $X$.  Thus $\mf M \models \wtr$.  Therefore $\mf M \models \rca + \wtr + \neg\dnrr{2}$.
\end{proof}

The principles $\mlrr{2}$ and $\dnrr{2}$ are closely related to the \emph{rainbow Ramsey theorem}.  Let $[\Nb]^n$ denote the set of $n$-element subsets of $\Nb$, and call a function $f \colon [\Nb]^n \imp \Nb$ \emph{$k$-bounded} if $|f^{-1}(c)| \leq k$ for every $c \in \Nb$.  Call an infinite $R \subseteq \Nb$ a \emph{rainbow for $f$} if $f$ is injective on $[R]^n$.  The rainbow Ramsey theorem for pairs and $2$-bounded colorings (denoted $\rrt^2_2$) is the statement ``for every 2-bounded  $f \colon [\Nb]^2 \imp \Nb$, there is a set $R$ that is a rainbow for $f$.''  By formalizing work of Csima and Mileti~\cite{CsimaMileti}, Conidis and Slaman~\cite{ConidisSlaman} have shown that $\rca \vdash \mlrr{2} \imp \rrt^2_2$.  J.\ Miller~\cite{MillerRRT}, again building on~\cite{CsimaMileti}, has shown that in fact $\rca \vdash \dnrr{2} \biimp \rrt^2_2$.  By Theorem~\ref{thm-WTRnot2DNR}, it follows that $\rca \nvdash \wtr \imp \rrt^2_2$.

From Theorem~\ref{thm-MLRimpBran}, we know that $\rca \vdash \mlr \imp \bran$.  In particular, if $h$ is any provably total function that is $O(2^n)$, then $\rca \vdash \mlr \imp \wdrr{h}$.  We now show that this implication is close to optimal.  Specifically, in Theorem~\ref{thm-WKLnothWDR} below we show that if $h$ is a provably total function that dominates the function $n \mapsto k^n$ for every $k$, then $\rca \nvdash \mlr \imp \wdrr{h}$.  In fact, in this case even $\wkl \nvdash \wdrr{h}$.  $\wkl$ is the system whose axioms are those of $\rca$, plus \emph{weak K\"onig's lemma}, which is the statement ``every infinite subtree of $2^{<\Nb}$ has an infinite path.''  $\wkl$ is strictly stronger than $\rca + \mlr$~\cite{YuSimpson}.

Recall the following definitions for a set $X \subseteq \omega$.
\begin{itemize}
\item Write $\sigma \leL X$ if $\sigma$ is to the left of $X$:  $\exists \rho(\rho^\smf 0 \subseteq \sigma \andd \rho^\smf 1 \subseteq X)$.  Then $X$ is \emph{left-r.e.}\ if $\{\sigma : \sigma \leL X\}$ is r.e.

\medskip

\item $X$ is \emph{superlow} if $X' \leqtt 0'$.  Equivalently, $X$ is superlow if $X' \leqwtt 0'$ because, for any $Z \subseteq \omega$, $Z \leqwtt 0'$ if and only if $Z \leqtt 0'$.
\end{itemize}

\begin{Proposition}\label{prop-SuperlowKre}
For every non-empty $\Pi^0_1$ class $\MP C \subseteq 2^\omega$, there is a superlow $Z \in \MP C$ such that, for every set $X \leqT Z$, there is a $k \in \omega$ such that $X$ is $k^n$-r.e.
\end{Proposition}

\begin{proof}
Let $Z \mapsto W^Z$ be the r.e.\ operator defined by
\begin{align*}
2^e(2n+1) \in W^Z \Biimp \Phi_e^Z(n) = 1.
\end{align*}
Apply the proof of the superlow basis theorem as given in~\cite[Theorem~1.8.38]{NiesBook}, but with the operator $W$ instead of the usual Turing jump operator $J$, to get a $Z \in \MP C$ such that $W^Z$ is left-r.e.  Clearly $Z' \leqm W^Z$, from which it follows that $Z$ superlow.  Now suppose that $X \leqT Z$, and let $e$ be such that $\Phi_e^Z = X$.  The fact that $W^Z$ is left-r.e.\ implies that $X$ is $2^{2^e(2n+1)}$-r.e., so $X$ is $k^n$-r.e.\ for $k = 2^{2^{e+2}}$.
\end{proof}

\begin{Proposition}\label{prop-WKL-Kre}
There is an $\omega$-model $\mf M = (\omega, \MP S)$ of $\wkl$ such that every $X \in \MP S$ is superlow and for every $X \in \MP S$ there is a $k \in \omega$ such that $X$ is $k^n$-r.e.
\end{Proposition} 
 
\begin{proof}
Given a set $Z$, decompose $Z$ into columns $Z = \bigoplus_{n \in \omega}Z_n$, and let $$\MP{S}_Z = \{X : \exists n(X \leqT \bigoplus_{i \leq n}Z_i)\}.$$  Let $\MP C \subseteq 2^\omega$ be a non-empty $\Pi^0_1$ class such that $(\omega, \MP{S}_Z) \models \wkl$ for all $Z \in \MP C$.  This can be accomplished, for example, by taking $\MP C$ to be the class of all sets $Z$ such that, for every $n$, $Z_{n+1}$ codes a $\{0,1\}$-valued diagonally non-recursive function relative to $\bigoplus_{i \leq n}Z_i$.  Then, for any such $Z$, $(\omega, \MP{S}_Z)$ models $\rca$ plus ``for every $X$ there is a $\{0,1\}$-valued diagonally non-recursive function relative to $X$,'' which is well-known to be equivalent to $\wkl$ by formalizing classic results of Jockusch and Soare~\cite{Jockusch-Soare:Pi01}.  Let $Z \in \MP C$ be as in the conclusion of Proposition~\ref{prop-SuperlowKre}.  Then every $X \in \MP{S}_Z$ is superlow, and, for every $X \in \MP{S}_Z$, there is a $k$ such that $X$ is $k^n$-r.e.  Thus $\mf M = (\omega, \MP{S}_Z)$ is the desired model.
\end{proof}

\begin{Theorem}\label{thm-WKLnothWDR}
Let $h \colon \omega \imp \omega$ be a function that is provably total in $\rca$ and eventually dominates the function $n \mapsto k^n$ for every $k \in \omega$.  Then there is an $\omega$-model
of $\wkl + \neg\wdrr{h}$.  Thus $\wkl \nvdash \wdrr{h}$ and therefore also $\rca \nvdash \mlr \imp \wdrr{h}$.
\end{Theorem}

\begin{proof}
If $X$ is $k^n$-r.e.\ and $h$ eventually dominates $k^n$, then it is straightforward to define an $h$-Demuth test capturing $X$.  Thus no $k^n$-r.e.\ set is $h$-weakly Demuth random.  Let $\mf M = (\omega, \MP S)$ be the model of $\wkl$ from Proposition~\ref{prop-WKL-Kre}.  Then no $X \in \MP S$ is $h$-weakly Demuth random because for every $X \in \MP S$ there is a $k$ such that $X$ is $k^n$-r.e.  Thus $\mf M \models \wkl + \neg\wdrr{h}$.
\end{proof}
 
\section*{Acknowledgements}
We thank David Belanger, Laurent Bienvenu, and Keita Yokoyama for helpful discussions.  We acknowledge the support of \emph{Centre International de Rencontres Math\'{e}matiques} and of \emph{Mathematisches Forschungsinstitut Oberwolfach}.  The first author acknowledges support through the Marsden fund of New Zealand. The second author   acknowledges the support of the \emph{Fonds voor Wetenschappelijk Onderzoek -- Vlaanderen} Pegasus program.  

\bibliographystyle{amsplain}
\bibliography{RMofRandomness}

\begin{bibdiv}
\begin{biblist}

\bib{Ambos.Kucera:00}{incollection}{
      author={Ambos-Spies, Klaus},
      author={Ku\v{c}era, Anton\'{\i}n},
       title={Randomness in computability theory},
        date={2000},
   booktitle={{C}omputability {T}heory and {I}ts {A}pplications: {C}urrent
  {T}rends and {O}pen {P}roblems ({B}oulder, {CO}, 1999)},
      series={Contemp. Math.},
      volume={257},
   publisher={Amer. Math. Soc., Providence, RI},
       pages={1\ndash 14},
         url={https://doi.org/10.1090/conm/257/04023},
      review={\MR{1770730}},
}

\bib{AvigadDeanRute}{article}{
      author={Avigad, Jeremy},
      author={Dean, Edward~T.},
      author={Rute, Jason},
       title={Algorithmic randomness, reverse mathematics, and the dominated
  convergence theorem},
        date={2012},
        ISSN={0168-0072},
     journal={Annals of Pure and Applied Logic},
      volume={163},
      number={12},
       pages={1854\ndash 1864},
         url={https://doi.org/10.1016/j.apal.2012.05.010},
      review={\MR{2964874}},
}

\bib{BarmpaliasDowneyNg}{article}{
      author={Barmpalias, George},
      author={Downey, Rod},
      author={Ng, Keng~Meng},
       title={Jump inversions inside effectively closed sets and applications
  to randomness},
        date={2011},
        ISSN={0022-4812},
     journal={Journal of Symbolic Logic},
      volume={76},
      number={2},
       pages={491\ndash 518},
         url={https://doi.org/10.2178/jsl/1305810761},
      review={\MR{2830414}},
}

\bib{Bauwens}{article}{
      author={Bauwens, Bruno},
       title={Prefix and plain {K}olmogorov complexity characterizations of
  2-randomness: simple proofs},
        date={2015},
        ISSN={0933-5846},
     journal={Archive for Mathematical Logic},
      volume={54},
      number={5-6},
       pages={615\ndash 629},
         url={https://doi.org/10.1007/s00153-015-0430-2},
      review={\MR{3372611}},
}

\bib{Bienvenu.Day.ea:14}{article}{
      author={Bienvenu, Laurent},
      author={Day, Adam~R.},
      author={Greenberg, Noam},
      author={Ku\v{c}era, Anton\'{\i}n},
      author={Miller, Joseph~S.},
      author={Nies, Andr\'{e}},
      author={Turetsky, Dan},
       title={Computing {$K$}-trivial sets by incomplete random sets},
        date={2014},
        ISSN={1079-8986},
     journal={Bulletin of Symbolic Logic},
      volume={20},
      number={1},
       pages={80\ndash 90},
         url={https://doi.org/10.1017/bsl.2013.3},
      review={\MR{3230826}},
}

\bib{BMSV}{article}{
      author={Bienvenu, Laurent},
      author={Muchnik, Andrej},
      author={Shen, Alexander},
      author={Vereshchagin, Nikolay},
       title={Limit complexities revisited},
        date={2010},
        ISSN={1432-4350},
     journal={Theory of Computing Systems},
      volume={47},
      number={3},
       pages={720\ndash 736},
         url={https://doi.org/10.1007/s00224-009-9203-9},
      review={\MR{2652038}},
}

\bib{BienvenuPateyShafer}{article}{
      author={Bienvenu, Laurent},
      author={Patey, Ludovic},
      author={Shafer, Paul},
       title={On the logical strengths of partial solutions to mathematical
  problems},
        date={2017},
        ISSN={2052-4986},
     journal={Transactions of the London Mathematical Society},
      volume={4},
      number={1},
       pages={30\ndash 71},
         url={https://doi.org/10.1112/tlm3.12001},
      review={\MR{3653054}},
}

\bib{Brattka2015LasVC}{incollection}{
      author={Brattka, Vasco},
      author={Gherardi, Guido},
      author={H\"{o}lzl, Rupert},
       title={Las {V}egas computability and algorithmic randomness},
        date={2015},
   booktitle={32nd {I}nternational {S}ymposium on {T}heoretical {A}spects of
  {C}omputer {S}cience},
      series={LIPIcs. Leibniz Int. Proc. Inform.},
      volume={30},
   publisher={Schloss Dagstuhl. Leibniz-Zent. Inform., Wadern},
       pages={130\ndash 142},
      review={\MR{3356408}},
}

\bib{Brattka2017UniformComputability}{article}{
      author={Brattka, Vasco},
      author={Hendtlass, Matthew},
      author={Kreuzer, Alexander~P.},
       title={On the uniform computational content of computability theory},
        date={2017},
        ISSN={1432-4350},
     journal={Theory of Computing Systems},
      volume={61},
      number={4},
       pages={1376\ndash 1426},
         url={https://doi.org/10.1007/s00224-017-9798-1},
      review={\MR{3712309}},
}

\bib{brattka2016algebraic}{article}{
      author={Brattka, Vasco},
      author={Pauly, Arno},
       title={On the algebraic structure of {W}eihrauch degrees},
        date={2018},
        ISSN={1860-5974},
     journal={Logical Methods in Computer Science},
      volume={14},
      number={4},
      review={\MR{3868998}},
}

\bib{Cholak:2006eg}{article}{
      author={Cholak, Peter},
      author={Greenberg, Noam},
      author={Miller, Joseph~S.},
       title={Uniform almost everywhere domination},
        date={2006},
        ISSN={0022-4812},
     journal={Journal of Symbolic Logic},
      volume={71},
      number={3},
       pages={1057\ndash 1072},
         url={https://doi.org/10.2178/jsl/1154698592},
      review={\MR{2251556}},
}

\bib{Conidis}{article}{
      author={Conidis, Chris~J.},
       title={Effectively approximating measurable sets by open sets},
        date={2012},
        ISSN={0304-3975},
     journal={Theoretical Computer Science},
      volume={428},
       pages={36\ndash 46},
         url={https://doi.org/10.1016/j.tcs.2012.01.011},
      review={\MR{2901377}},
}

\bib{ConidisSlaman}{article}{
      author={Conidis, Chris~J.},
      author={Slaman, Theodore~A.},
       title={Random reals, the rainbow {R}amsey theorem, and arithmetic
  conservation},
        date={2013},
        ISSN={0022-4812},
     journal={Journal of Symbolic Logic},
      volume={78},
      number={1},
       pages={195\ndash 206},
         url={https://doi.org/10.2178/jsl.7801130},
      review={\MR{3087070}},
}

\bib{CsimaMileti}{article}{
      author={Csima, Barbara~F.},
      author={Mileti, Joseph~R.},
       title={The strength of the rainbow {R}amsey theorem},
        date={2009},
        ISSN={0022-4812},
     journal={Journal of Symbolic Logic},
      volume={74},
      number={4},
       pages={1310\ndash 1324},
         url={https://doi.org/10.2178/jsl/1254748693},
      review={\MR{2583822}},
}

\bib{Day.Miller:15}{article}{
      author={Day, Adam~R.},
      author={Miller, Joseph~S.},
       title={Density, forcing, and the covering problem},
        date={2015},
        ISSN={1073-2780},
     journal={Mathematical Research Letters},
      volume={22},
      number={3},
       pages={719\ndash 727},
         url={https://doi.org/10.4310/MRL.2015.v22.n3.a5},
      review={\MR{3350101}},
}

\bib{Demuth:75a}{article}{
      author={Demuth, Osvald},
       title={Constructive pseudonumbers},
        date={1975},
        ISSN={0010-2628},
     journal={Commentationes Mathematicae Universitatis Carolinae},
      volume={16},
       pages={315\ndash 331},
        note={(Russian)},
      review={\MR{0381955}},
}

\bib{DHBook}{book}{
      author={Downey, Rodney~G.},
      author={Hirschfeldt, Denis~R.},
       title={Algorithmic randomness and complexity},
      series={Theory and Applications of Computability},
   publisher={Springer, New York},
        date={2010},
        ISBN={978-0-387-95567-4},
         url={https://doi.org/10.1007/978-0-387-68441-3},
      review={\MR{2732288}},
}

\bib{FigueiraHirschfeldtMillerNgNies}{article}{
      author={Figueira, Santiago},
      author={Hirschfeldt, Denis~R.},
      author={Miller, Joseph~S.},
      author={Ng, Keng~Meng},
      author={Nies, Andr\'{e}},
       title={Counting the changes of random {$\Delta_2^0$} sets},
        date={2015},
        ISSN={0955-792X},
     journal={Journal of Logic and Computation},
      volume={25},
      number={4},
       pages={1073\ndash 1089},
         url={https://doi.org/10.1093/logcom/exs083},
      review={\MR{3449891}},
}

\bib{Franklin.Ng:10}{article}{
      author={Franklin, Johanna N.~Y.},
      author={Ng, Keng~Meng},
       title={Difference randomness},
        date={2011},
        ISSN={0002-9939},
     journal={Proceedings of the American Mathematical Society},
      volume={139},
      number={1},
       pages={345\ndash 360},
         url={https://doi.org/10.1090/S0002-9939-2010-10513-0},
      review={\MR{2729096}},
}

\bib{Friedman}{inproceedings}{
      author={Friedman, Harvey},
       title={Some systems of second order arithmetic and their use},
        date={1975},
   booktitle={Proceedings of the {I}nternational {C}ongress of {M}athematicians
  ({V}ancouver, {B}. {C}., 1974), {V}ol. 1},
       pages={235\ndash 242},
      review={\MR{0429508}},
}

\bib{HajekPudlak}{book}{
      author={H\'{a}jek, Petr},
      author={Pudl\'{a}k, Pavel},
       title={{M}etamathematics of {F}irst-{O}rder {A}rithmetic},
      series={Perspectives in Mathematical Logic},
   publisher={Springer-Verlag, Berlin},
        date={1998},
        ISBN={3-540-63648-X},
        note={Second printing},
      review={\MR{1748522}},
}

\bib{HinmanSurvey}{article}{
      author={Hinman, Peter~G.},
       title={A survey of {M}u\v{c}nik and {M}edvedev degrees},
        date={2012},
        ISSN={1079-8986},
     journal={Bulletin of Symbolic Logic},
      volume={18},
      number={2},
       pages={161\ndash 229},
         url={https://doi.org/10.2178/bsl/1333560805},
      review={\MR{2931672}},
}

\bib{Jockusch-Soare:Pi01}{article}{
      author={Jockusch, Carl~G., Jr.},
      author={Soare, Robert~I.},
       title={{$\Pi ^{0}_{1}$} classes and degrees of theories},
        date={1972},
        ISSN={0002-9947},
     journal={Transactions of the American Mathematical Society},
      volume={173},
       pages={33\ndash 56},
         url={https://doi.org/10.2307/1996261},
      review={\MR{316227}},
}

\bib{Kurtz:81}{thesis}{
      author={Kurtz, Stuart~Alan},
       title={{R}andomness and {G}enericity in the {D}egrees of
  {U}nsolvability},
        type={Ph.D. Thesis},
   publisher={ProQuest LLC, Ann Arbor, MI},
        date={1981},
  url={http://gateway.proquest.com/openurl?url_ver=Z39.88-2004&rft_val_fmt=info:ofi/fmt:kev:mtx:dissertation&res_dat=xri:pqdiss&rft_dat=xri:pqdiss:8203509},
        note={University of Illinois at Urbana-Champaign},
      review={\MR{2631709}},
}

\bib{Kucera:85}{incollection}{
      author={Ku\v{c}era, Anton\'{\i}n},
       title={Measure, {$\Pi^0_1$}-classes and complete extensions of {${\rm
  PA}$}},
        date={1985},
   booktitle={{R}ecursion {T}heory {W}eek ({O}berwolfach, 1984)},
      series={Lecture Notes in Math.},
      volume={1141},
   publisher={Springer, Berlin},
       pages={245\ndash 259},
         url={https://doi.org/10.1007/BFb0076224},
      review={\MR{820784}},
}

\bib{Kucera:86}{incollection}{
      author={Ku\v{c}era, Anton\'{\i}n},
       title={An alternative, priority-free, solution to {P}ost's problem},
        date={1986},
   booktitle={{M}athematical {F}oundations of {C}omputer {S}cience, 1986
  ({B}ratislava, 1986)},
      series={Lecture Notes in Comput. Sci.},
      volume={233},
   publisher={Springer, Berlin},
       pages={493\ndash 500},
         url={https://doi.org/10.1007/BFb0016275},
      review={\MR{874627}},
}

\bib{Kucera.Nies:11}{article}{
      author={Ku\v{c}era, Anton\'{\i}n},
      author={Nies, Andr\'{e}},
       title={Demuth randomness and computational complexity},
        date={2011},
        ISSN={0168-0072},
     journal={Annals of Pure and Applied Logic},
      volume={162},
      number={7},
       pages={504\ndash 513},
         url={https://doi.org/10.1016/j.apal.2011.01.004},
      review={\MR{2781092}},
}

\bib{MartinLof:68}{inproceedings}{
      author={Martin-L\"{o}f, Per},
       title={On the notion of randomness},
        date={1970},
   booktitle={Intuitionism and {P}roof {T}heory ({P}roc. {C}onf., {B}uffalo,
  {N}.{Y}., 1968)},
   publisher={North-Holland, Amsterdam},
       pages={73\ndash 78},
      review={\MR{0275483}},
}

\bib{MillerRRT}{unpublished}{
      author={Miller, Joseph~S.},
       title={personal communication},
}

\bib{MillerKRand}{article}{
      author={Miller, Joseph~S.},
       title={Every 2-random real is {K}olmogorov random},
        date={2004},
        ISSN={0022-4812},
     journal={Journal of Symbolic Logic},
      volume={69},
      number={3},
       pages={907\ndash 913},
         url={https://doi.org/10.2178/jsl/1096901774},
      review={\MR{2078929}},
}

\bib{Miller.Nies:06}{article}{
      author={Miller, Joseph~S.},
      author={Nies, Andr\'{e}},
       title={Randomness and computability: {O}pen questions},
        date={2006},
        ISSN={1079-8986},
     journal={Bulletin of Symbolic Logic},
      volume={12},
      number={3},
       pages={390\ndash 410},
         url={http://projecteuclid.org/euclid.bsl/1154698740},
      review={\MR{2248590}},
}

\bib{Miyabe}{inproceedings}{
      author={Miyabe, Kenshi},
       title={Muchnik degrees and {M}edvedev degrees of randomness notions},
        date={2019},
   booktitle={Proceedings of the 14th and 15th {A}sian {L}ogic {C}onferences},
   publisher={World Sci. Publ., Hackensack, NJ},
       pages={108\ndash 128},
      review={\MR{3890460}},
}

\bib{Miyabe.Nies.Zhang:16}{article}{
      author={Miyabe, Kenshi},
      author={Nies, Andr\'{e}},
      author={Zhang, Jing},
       title={Using almost-everywhere theorems from analysis to study
  randomness},
        date={2016},
        ISSN={1079-8986},
     journal={Bulletin of Symbolic Logic},
      volume={22},
      number={3},
       pages={305\ndash 331},
         url={https://doi.org/10.1017/bsl.2016.10},
      review={\MR{3569255}},
}

\bib{NiesBook}{book}{
      author={Nies, Andr\'{e}},
       title={{C}omputability and {R}andomness},
      series={Oxford Logic Guides},
   publisher={Oxford University Press, Oxford},
        date={2009},
      volume={51},
        ISBN={978-0-19-923076-1},
         url={https://doi.org/10.1093/acprof:oso/9780199230761.001.0001},
      review={\MR{2548883}},
}

\bib{NiesStephanTerwijn}{article}{
      author={Nies, Andr\'{e}},
      author={Stephan, Frank},
      author={Terwijn, Sebastiaan~A.},
       title={Randomness, relativization and {T}uring degrees},
        date={2005},
        ISSN={0022-4812},
     journal={Journal of Symbolic Logic},
      volume={70},
      number={2},
       pages={515\ndash 535},
         url={https://doi.org/10.2178/jsl/1120224726},
      review={\MR{2140044}},
}

\bib{NiesTriplettYokoyama}{unpublished}{
      author={Nies, Andr\'{e}},
      author={Triplett, Marcus~A.},
      author={Yokoyama, Keita},
       title={The reverse mathematics of theorems of {J}ordan and {L}ebesgue},
        date={2017},
        note={preprint, arXiv:1704.00931},
}

\bib{Schnorr:71}{article}{
      author={Schnorr, Claus-Peter},
       title={A unified approach to the definition of random sequences},
        date={1971},
        ISSN={0025-5661},
     journal={Mathematical Systems Theory. An International Journal on
  Mathematical Computing Theory},
      volume={5},
       pages={246\ndash 258},
         url={https://doi.org/10.1007/BF01694181},
      review={\MR{0354328}},
}

\bib{Schnorr:75}{book}{
      author={Schnorr, Claus-Peter},
       title={Zuf\"{a}lligkeit und {W}ahrscheinlichkeit. {E}ine algorithmische
  {B}egr\"{u}ndung der {W}ahrscheinlichkeitstheorie},
      series={Lecture Notes in Mathematics, Vol. 218},
   publisher={Springer-Verlag, Berlin-New York},
        date={1971},
      review={\MR{0414225}},
}

\bib{SimpsonMassProblemsRandomness}{article}{
      author={Simpson, Stephen~G.},
       title={Mass problems and randomness},
        date={2005},
        ISSN={1079-8986},
     journal={Bulletin of Symbolic Logic},
      volume={11},
      number={1},
       pages={1\ndash 27},
         url={https://doi.org/10.2178/bsl/1107959497},
      review={\MR{2125147}},
}

\bib{SimpsonSOSOA}{book}{
      author={Simpson, Stephen~G.},
       title={{S}ubsystems of {S}econd {O}rder {A}rithmetic},
     edition={2},
      series={Perspectives in Logic},
   publisher={Cambridge University Press, Cambridge; Association for Symbolic
  Logic, Poughkeepsie, NY},
        date={2009},
        ISBN={978-0-521-88439-6},
         url={https://doi.org/10.1017/CBO9780511581007},
      review={\MR{2517689}},
}

\bib{SimpsonSurvey}{article}{
      author={Simpson, Stephen~G.},
       title={Mass problems associated with effectively closed sets},
        date={2011},
        ISSN={0040-8735},
     journal={Tohoku Mathematical Journal. Second Series},
      volume={63},
      number={4},
       pages={489\ndash 517},
         url={https://doi.org/10.2748/tmj/1325886278},
      review={\MR{2872953}},
}

\bib{SlamanInd}{article}{
      author={Slaman, Theodore~A.},
       title={{$\Sigma_n$}-bounding and {$\Delta_n$}-induction},
        date={2004},
        ISSN={0002-9939},
     journal={Proceedings of the American Mathematical Society},
      volume={132},
      number={8},
       pages={2449\ndash 2456},
         url={https://doi.org/10.1090/S0002-9939-04-07294-6},
      review={\MR{2052424}},
}

\bib{SlamanFrag}{unpublished}{
      author={Slaman, Theodore~A.},
       title={The first-order fragments of second-order theories},
        date={2011},
        note={CiE 2011},
}

\bib{YuLebesgue}{article}{
      author={Yu, Xiaokang},
       title={Lebesgue convergence theorems and reverse mathematics},
        date={1994},
        ISSN={0942-5616},
     journal={Mathematical Logic Quarterly},
      volume={40},
      number={1},
       pages={1\ndash 13},
         url={https://doi.org/10.1002/malq.19940400102},
      review={\MR{1284435}},
}

\bib{YuSimpson}{article}{
      author={Yu, Xiaokang},
      author={Simpson, Stephen~G.},
       title={Measure theory and weak {K}\"{o}nig's lemma},
        date={1990},
        ISSN={0933-5846},
     journal={Archive for Mathematical Logic},
      volume={30},
      number={3},
       pages={171\ndash 180},
         url={https://doi.org/10.1007/BF01621469},
      review={\MR{1080236}},
}

\end{biblist}
\end{bibdiv}

\vfill

\end{document}